\documentclass[a4paper]{article}

\usepackage{etoolbox}
\usepackage[utf8]{inputenc}
\usepackage{enumerate}
\usepackage{subfigure}
\usepackage{pgfplots} 
\usepackage{sectsty}
\usepackage{float} 
\usepackage{amssymb}
\usepackage{amsmath}
\usepackage{amsthm}
\usepackage{MnSymbol} 
\usepackage{bbm}
\usepackage[pdfpagelabels]{hyperref}
\usepackage[all]{hypcap} 
\usepackage{pi}
\usepackage{todonotes}

\newtoggle{colored links}
\newtoggle{colored symbols}
\newtoggle{colored pictures}
\newtoggle{decorations}

\toggletrue{colored links}
\toggletrue{colored symbols}
\toggletrue{colored pictures}
\toggletrue{decorations}

\iftoggle{colored pictures}{
  \edef\tikzOptions{colors=true}
}{
  \edef\tikzOptions{colors=false}
}
\iftoggle{decorations}{
  \edef\tikzOptions{\tikzOptions,decorations=true}
}{
  \edef\tikzOptions{\tikzOptions,decorations=false}
}
\iftoggle{colored links}{
  \hypersetup{colorlinks,citecolor=blue,linkcolor=blue,urlcolor=blue}
}{
  \hypersetup{colorlinks,citecolor=black,linkcolor=black,urlcolor=black}
}
\piLoadConfigurations

\sectionfont{\large}

\newtheorem{theorem}{Theorem}[section]
\newtheorem{corollary}[theorem]{Corollary}

\newtheorem{lemma}[theorem]{Lemma}

\newtheorem{proposition}[theorem]{Proposition}

\newtheorem{example}[theorem]{Example}
\newtheorem{question}[theorem]{Question}

\DeclareDocumentCommand\orderO{m}{\mathcal{O}\left(#1\right)}
\DeclareDocumentCommand\orderOmega{m}{\Omega\left(#1\right)}
\DeclareDocumentCommand\ordero{m}{o\left(#1\right)}
\DeclareDocumentCommand\orderTheta{m}{\Theta\left(#1\right)}
\DeclareDocumentCommand\setdef{mo}{\left\{#1\IfNoValueTF{#2}{}{ : #2}\right\}}
\DeclareDocumentCommand\bincoeff{mm}{\genfrac{(}{)}{0pt}{}{#1}{#2}}
\DeclareDocumentCommand\conv{o}{\operatorname{conv}\IfValueTF{#1}{\left(#1\right)}{}}
\DeclareDocumentCommand\onevec{o}{\IfNoValueTF{#1}{\mathbbm{1}}{\mathbbm{1}_{#1}}}
\DeclareDocumentCommand\proj{oo}{\IfValueTF{#1}{\operatorname{proj}{}_{#1}}{%
  \operatorname{proj}{}}\IfValueTF{#2}{\left(#2\right)}{}}
\DeclareDocumentCommand\scalprod{mm}{\left<#1,#2\right>}
\DeclareDocumentCommand\unitvec{m}{\mathbbm{e}^{#1}}
\DeclareDocumentCommand\zerovec{o}{\IfNoValueTF{#1}{\mathbb{O}}{\mathbb{O}_{#1}}}
\DeclareDocumentCommand\cplxNP{}{\mathsf{NP}}
\DeclareDocumentCommand\dim{o}{\operatorname{dim}\IfValueTF{#1}{\left(#1\right)}{}}
\DeclareDocumentCommand\homog{o}{\operatorname{homog}\IfValueTF{#1}{\left(#1\right)}{}}
\DeclareDocumentCommand\cl{o}{\operatorname{cl}\IfValueTF{#1}{\left(#1\right)}{}}
\DeclareDocumentCommand\cone{o}{\operatorname{cone}\IfValueTF{#1}{\left(#1\right)}{}}
\DeclareDocumentCommand\lin{o}{\operatorname{span}\IfValueTF{#1}{\left(#1\right)}{}}

\DeclareDocumentCommand\xc{o}{\operatorname{xc}%
  \IfValueTF{#1}{\left(#1\right)}{}}
\DeclareDocumentCommand\polyLattice{o}{\mathcal{L}\hspace{-.1em}\IfValueTF{#1}{\left(#1\right)}{}}

\DeclareDocumentCommand\N{}{\mathbb{N}}
\DeclareDocumentCommand\R{}{\mathbb{R}}

\DeclareDocumentCommand\BicliqueFaces{o}{%
  \mathcal{F}\IfValueTF{#1}{\hspace{-.1em}\left(#1\right)}{}}
\DeclareDocumentCommand\BicliqueVertices{o}{%
  \mathcal{V}\IfValueTF{#1}{\hspace{-.1em}\left(#1\right)}{}}
\DeclareDocumentCommand\BicliqueGraph{o}{%
  G_{\mathcal{N}}\IfValueTF{#1}{\hspace{-.1em}\left(#1\right)}{}}
\DeclareDocumentCommand\sxc{o}{\operatorname{sxc}\IfNoValueTF{#1}{}{\left(#1\right)}}
\DeclareDocumentCommand\CommonNeighbors{o}{\Lambda\IfValueTF{#1}{\left(#1\right)}{}}

\NewDocumentCommand\PerfMatchPoly{m}{\ensuremath{P^{\text{perf}}_{\text{match}}\left( #1 \right)}}
\NewDocumentCommand\MatchPoly{m}{\ensuremath{P_{\text{match}}\left( #1 \right)}}

\iftoggle{colored symbols}
{
\NewDocumentCommand\pmMone{}{\textcolor{red}{\ensuremath{M_1}}}
\NewDocumentCommand\pmMtwo{}{\textcolor{green!70!black}{\ensuremath{M_2}}}
\NewDocumentCommand\pmMthree{}{\textcolor{orange}{\ensuremath{M_3}}}
\NewDocumentCommand\pmMprime{}{\textcolor{blue}{\ensuremath{M'}}}
\NewDocumentCommand\pmMbar{}{\textcolor{blue}{\ensuremath{\overline{M}}}}
\NewDocumentCommand\pmMstar{}{\textcolor{blue}{\ensuremath{M^*}}}
\NewDocumentCommand\pmMj{}{\textcolor{red}{\ensuremath{M_j}}}
\NewDocumentCommand\pmMk{}{\textcolor{green!70!black}{\ensuremath{M_k}}}
}{
\NewDocumentCommand\pmMone{}{\ensuremath{M_1}}
\NewDocumentCommand\pmMtwo{}{\ensuremath{M_2}}
\NewDocumentCommand\pmMthree{}{\ensuremath{M_3}}
\NewDocumentCommand\pmMprime{}{\ensuremath{M'}}
\NewDocumentCommand\pmMbar{}{\ensuremath{\overline{M}}}
\NewDocumentCommand\pmMstar{}{\ensuremath{M^*}}
\NewDocumentCommand\pmMj{}{\ensuremath{M_j}}
\NewDocumentCommand\pmMk{}{\ensuremath{M_k}}
}

\NewDocumentCommand\SpanTreePoly{m}{\ensuremath{P_{\textrm{spt}}\left( #1 \right)}}
\NewDocumentCommand\sptAllW{}{\ensuremath{\mathcal{W}}}
\iftoggle{colored symbols}
{
\NewDocumentCommand\sptTe{}{\textcolor{red}{\ensuremath{e}}}
\NewDocumentCommand\sptTf{}{\textcolor{green!70!black}{\ensuremath{f}}}
\NewDocumentCommand\sptTg{}{\textcolor{blue}{\ensuremath{g}}}
\NewDocumentCommand\sptTh{}{\textcolor{orange}{\ensuremath{h}}}
\NewDocumentCommand\sptTj{}{\textcolor{purple!70!black}{\ensuremath{j}}}

\NewDocumentCommand\sptTW{}{\ensuremath{\textcolor{blue}{T}(W)}}
\NewDocumentCommand\sptTprime{}{\textcolor{red}{\ensuremath{T'}}}
}{ 
\NewDocumentCommand\sptTe{}{\ensuremath{e}}
\NewDocumentCommand\sptTf{}{\ensuremath{f}}
\NewDocumentCommand\sptTg{}{\ensuremath{g}}
\NewDocumentCommand\sptTh{}{\ensuremath{h}}
\NewDocumentCommand\sptTj{}{\ensuremath{j}}

\NewDocumentCommand\sptTW{}{\ensuremath{T(W)}}
\NewDocumentCommand\sptTprime{}{\ensuremath{T'}}
}

\iftoggle{colored symbols}
{
\NewDocumentCommand\khsi{}{\textcolor{red}{\ensuremath{i}}}
\NewDocumentCommand\khsj{}{\textcolor{blue}{\ensuremath{j}}}
\NewDocumentCommand\khsv{}{\textcolor{green!70!black}{\ensuremath{v}}}
}{ 
\NewDocumentCommand\khsi{}{\ensuremath{i}}
\NewDocumentCommand\khsj{}{\ensuremath{j}}
\NewDocumentCommand\khsv{}{\ensuremath{v}}
}

\NewDocumentCommand\FlowPoly{om}{%
  \ensuremath{P_{\IfNoValueTF{#1}{}{\textrm{#1-}}\textrm{flow}}\left( #2 \right)}}
\NewDocumentCommand\FlowPaths{om}{%
  \ensuremath{ \IfNoValueTF{#1}{\mathcal{P}}{\mathcal{P}_{ #1 }}\left( #2 \right) }}

\iftoggle{colored symbols}
{
\NewDocumentCommand\flowB{}{\textcolor{orange}{\ensuremath{B}}}
\NewDocumentCommand\flowP{}{\textcolor{blue}{\ensuremath{P}}}
\NewDocumentCommand\flowPprime{}{\textcolor{green!70!black}{\ensuremath{P'}}}
\NewDocumentCommand\flowW{}{\textcolor{red}{\ensuremath{W}}}
\NewDocumentCommand\flowWprime{}{\textcolor{purple!70!black}{\ensuremath{W'}}}
}{ 
\NewDocumentCommand\flowB{}{\ensuremath{B}}
\NewDocumentCommand\flowP{}{\ensuremath{P}}
\NewDocumentCommand\flowPprime{}{\ensuremath{P'}}
\NewDocumentCommand\flowW{}{\ensuremath{W}}
\NewDocumentCommand\flowWprime{}{\ensuremath{W'}}
}

\title{Simple Extensions of Polytopes}
\author{Volker Kaibel and Matthias Walter\\
Otto-von-Guericke University Magdeburg}

\begin{document}

\maketitle

\begin{abstract}
We introduce the \emph{simple extension complexity} of
a polytope $P$ as the smallest number of facets of any
simple (i.e., non-degenerate in the sense of linear programming) polytope which can be projected onto $P$.
We devise a combinatorial method to establish lower bounds on the simple extension complexity and show 
for several polytopes that they have large simple extension complexities.
These examples include both the spanning tree and the perfect matching polytopes of complete graphs,
uncapacitated flow polytopes for non-trivially decomposable directed acyclic graphs,
hypersimplices, and random 0/1-polytopes with vertex numbers within a certain range.
On our way to obtain the result on perfect matching polytopes we generalize a result of Padberg and Rao's
on the adjacency structures of those polytopes.
To complement the lower bounding techniques
we characterize in which cases known construction techniques
yield simple extensions.
\end{abstract}

\pagebreak[3]
\section{Introduction}

In combinatorial optimization, linear programming formulations are a standard tool 
to gain structural insight, derive algorithms and to analyze computational complexity.
With respect to both structural and algorithmic aspects linear optimization 
over a polytope~$P$ can be replaced by linear optimization over any 
(usually higher dimensional) polytope~$Q$ of which~$P$ can be obtained as the image 
under a linear map (which we refer to as a \emph{projection}). 
Such a polytope~$Q$ (along with a suitable projection) is called an \emph{extension} of~$P$. 

Defining the \emph{size} of a polytope as its number of facets,
the smallest size of any extension of  the polytope $P$ is known as 
the \emph{extension complexity} $\xc[P]$ of $P$. 
It has turned out in the past that for several important polytopes
related to combinatorial optimization problems the  extension complexity
is bounded polynomially in the dimension.
One of the most prominent examples is the spanning tree polytope of the complete
graph~$K_n$ on $n$ nodes, which has extension complexity $\orderO{n^3}$~\cite{Martin91}. 

After Rothvoß~\cite{Rothvoss12} showed that there are 0/1-polytopes whose extension 
complexities cannot be bounded polynomially in their dimensions, 
only recently Fiorini, Massar, Pokutta, Tiwary and de Wolf \cite{FioriniMPTW12} could prove
that the extension complexities of some concrete and important examples of polytopes like
traveling salesman polytopes cannot be bounded polynomially.
Similar results have then also been deduced for several other polytopes associated with
$\cplxNP$-hard optimization problems, e.g.,
by Avis and Tiwary~\cite{AvisT13} and Pokutta and van Vyve~\cite{PokuttaV13}. 
Very recently, Rothvoß~\cite{Rothvoss14} showed that also the perfect matching polytope of the complete 
graph (with an even number of nodes) has exponential extension complexity, thus exhibiting the 
first polytope with this property that is associated with a polynomial time solvable optimization problem.  

The first fundamental research with respect to understanding extension complexities was 
Yannakakis' seminal paper~\cite{Yannakakis91} of 1991. Observing that many of the nice and 
small extensions that are known (e.g., the polynomial size extension of the spanning
tree polytope of $K_n$ mentioned above) have the nice property of being symmetric in a certain sense, 
he derived lower bounds on extensions with that special property. In particular, he already
proved that both perfect matching polytopes as well as traveling salesman polytopes do not
have polynomial size \emph{symmetric} extensions. 

It turned out that requiring symmetry in principle 
actually can make a huge difference for the minimum sizes of extensions (though nowadays we know that this is not really true for traveling salesman and perfect matching polytopes).
For instance, Kaibel, Theis, and Pashkovich~\cite{KaibelPT12} showed that the
polytope associated with the matchings of size $\lfloor \log n\rfloor$ in $K_n$
has polynomially bounded extension complexity although it does not admit symmetric extensions
of polynomial size.
Another example is provided by the permutahedron which has extension 
complexity $\orderTheta{n\log n}$~\cite{Goemans09},
while every symmetric extension of it has size $\orderOmega{n^2}$~\cite{Pashkovich14}. 

These examples show that imposing the restriction of symmetry may severely influence
the smallest possible sizes of extensions.
In this paper, we investigate another type of restrictions on extensions,
namely the one arising from requiring the extension to be a non-degenerate polytope. 
A $d$-dimensional polytope is called \emph{simple} if every vertex
is contained in exactly $d$~facets.
We denote by $\sxc[P]$ the \emph{simple extension complexity}, i.e.,
the smallest size of any simple extension of the polytope~$P$. 

From a practical point of view, simplicity is an interesting property
since it formalizes primal non-degeneracy of linear programs.
In addition, large parts of combinatorial/extremal theory of polytopes
deal with simple polytopes.
Furthermore, as with other restrictions like symmetry, there indeed exist
nice examples of simple extensions of certain polytopes relevant in optimization.
For instance, generalizing the well-known fact that the permutahedron is a zonotope, Wolsey showed in the late 80's (personal communication) that, for arbitrary processing times,
the completion time polytope for $n$ jobs is a projection of an $\orderO{n^2}$-dimensional cube.
The main results of this paper show, however, that for several polytopes relevant in optimization
(among them both perfect matching polytopes and spanning tree polytopes) insisting on simplicity
enforces very large sizes of the extensions.
More precisely, we establish that for the following polytopes the simple extension complexity
equals their number of vertices (note that the number of vertices of~$P$ is a trivial
upper bound for $\sxc[P]$, realized by the extension obtained from
writing $P$ as the convex hull of its vertices): 
\begin{itemize}
\item
  Perfect matching polytopes of complete graphs
  (Theorem~\ref{TheoremPerfectMatchingPolytopeSimpleExtensionComplexity})
\item
  Uncapacitated flow polytopes of non-decomposable acyclic networks 
  (Theorem~\ref{TheoremFlowSimpleExtensionComplexity})
\item
  (Certain) random 0/1-polytopes  (Theorem~\ref{TheoremRandomPolytopesSimpleExtensionComplexity})
\item
  Hypersimplices (Theorem~\ref{TheoremHypersimplex})
\end{itemize}
Furthermore, we prove that
\begin{itemize}
	\item 
    the spanning tree polytope of the complete graph with~$n$ nodes
    has simple extension complexity at least
    $\orderOmega{2^{n-\ordero{n}}}$
    (Theorem~\ref{TheoremSpanningTreePolytopeSimpleExtensionComplexity}).
\end{itemize}

The paper is structured as follows:
We first focus on known construction techniques and characterize when
reflections and disjunctive programming yield simple extensions (Section~\ref{SectionConstructions}).
We continue with some techniques to bound the simple extension complexity of a polytope
from below (Section~\ref{SectionTechniques}).
Then we deduce our results on hypersimplices (Section~\ref{SectionHypersimplex}),
spanning tree polytopes (Section~\ref{SectionSpanningTreePolytope}),
flow polytopes (Section~\ref{SectionFlowPolytopes}),
and perfect matching polytopes (Section~\ref{SectionPerfectMatchingPolytope}).
The core of the latter part is a strengthening of a result of
Padberg and Rao's \cite{PadbergR74} on adjacencies in the perfect
matching polytope (Theorem~\ref{TheoremThreeCommonNeighbor}), which may be of independent interest. 

Let us end this introduction by remarking that the concept of \emph{simplicial extensions}
is not interesting.
To see this, observe that any $d$-polytope $Q$ with $N$ vertices has at least $d \cdot N$ facet-vertex incidences since
every vertex lies in at least $d$ facets.
On the other hand, if $Q$ is simplicial (i.e., all facets are simplices) and has $f$ facets, the number of facet-vertex incidences
is equal to $d \cdot f$, proving $f \geq N$.
For every polytope $P$ with $N$ vertices, every extension polytope has at least $N$ vertices,
and hence the smallest possible simplicial
extension polytope of $P$ is the simplex with $N$ vertices.

\pagebreak[3]
\section{Constructions}
\label{SectionConstructions}

There are three major techniques for constructing extended formulations,
namely dynamic programming, disjunctive programming, and reflections.
Extensions based on dynamic programs yield network flow polytopes
for acyclic graphs which are not simple in general
and also have large simple extension complexities
(see Section~\ref{SectionFlowPolytopes}).

In this section we characterize for the other two techniques mentioned above
in which cases the produced extensions are simple.

\subsection{Reflections}

Let $P = \setdef{ x \in \R^n }[ Ax \leq b]$ be a polytope
and let $H_{\leq} = \setdef{ x \in \R^n }[ \scalprod{a}{x} \leq \beta ]$
be a halfspace in $\R^n$.
Denoting by $P_1 := P \cap H_{\leq}$ the intersection of the polytope with the halfspace
and by $P_2$ the image of $P_1$ under reflection at the boundary hyperplane
$H_=$ of $H_{\leq}$,
we call $\conv(P_1 \cup P_2)$ the \emph{reflection} of $P$ at $H_{\leq}$.
The technique in \cite{KaibelP11a} provides an extended formulation
for this polytope.

\begin{proposition}[Kaibel \& Pashkovich \cite{KaibelP11a}]
  \label{TheoremReflectionExtension}
  The polytope $Q$ defined by
  \begin{gather*}
    Q = \setdef{ (x,y) \in \R^{n + n} }[ 
      Ay \leq b,
      \scalprod{a}{y} \leq \scalprod{a}{x} \leq 2 \beta - \scalprod{a}{y},
      (x-y) \in \lin[a] ]
  \end{gather*}
  together with the projection onto the $x$-space
  is an extension of $\conv[P_1 \cup P_2]$.
\end{proposition}

Our contribution is the next theorem which 
clarifies under which circumstances $Q$ is a simple polytope.

\begin{theorem}
  \label{TheoremReflectionSimple}
  Let $Q$ be the extension polytope for the reflection of $P$ at $H_{\leq}$
  as defined in this subsection,
  let $P_1 := P \cap H_{\leq}$, and let $F := P_1 \cap H_=$ be the
  intersection of $P_1$ with the reflection hyperplane.

  Then $Q$ is simple if and only if $P_1$ is simple and either $P_1 = F$,
  or $F$ is a facet of $P_1$ or $F = \emptyset$.
\end{theorem}

\begin{proof}
  We first observe that the faces
  \begin{align*}
    Q_1 &:= Q \cap \setdef{ (x,y) \in \R^n \times \R^n }[ \scalprod{a}{y} = \scalprod{a}{x} ] \\
    Q_2 &:= Q \cap \setdef{ (x,y) \in \R^n \times \R^n }[ \scalprod{a}{x} = 2 \beta - \scalprod{a}{y} ]
  \end{align*}
  of $Q$ are both affinely isomorphic to $P_1$.
  Thus $Q$ can only be simple if $P_1$ is so.
  If $P_1 \subseteq H_=$ holds, $Q = Q_1 = Q_2$ and $Q$ is simple if and only if $P_1$ is simple,
  proving the equivalence in case $P_1 = F$.
  Otherwise, let $d := \dim P_1$ and observe that $\dim Q = d+1$ holds
  because $Q_1$ and $Q_2$ are proper faces of $Q$ and $Q$'s
  dimension cannot be larger than $d+1$.
  Furthermore, $(x,y) \in Q_1 \cap Q_2$ holds if and only if $\scalprod{a}{x} = \beta$ is satisfied,
  hence $Q_1 \cap Q_2$ is affinely isomorphic to $F$.
  Define $k := \dim F$.

  We now assume that $Q$ is simple and $F \neq \emptyset$, i.e., $k \geq 0$ holds.
  Let $v$ be any vertex of $Q_1 \cap Q_2$.
  Since $Q$ is simple and of dimension $d+1$, $v$ has $d+1$ adjacent vertices,
  $k$ of which lie in $Q_1 \cap Q_2$ (isomorphic to $F$).
  Furthermore, $v$ has $d$ neighbors in $Q_i$ for $i =1,2$.
  Hence, $k$ of these vertices lie in $Q_1 \cap Q_2$,
  $d-k$ lie in $Q_1 \setminus Q_2$ and $d-k$ lie in $Q_2 \setminus Q_1$.
  The resulting equation $k + (d-k) + (d-k) = d+1$ yields $k = d-1$,
  i.e., $F$ is a facet of $P_1$.
  This proves necessity of the condition.

  To prove sufficiency, from now on assume that $P_1$ is simple and $\dim Q = d+1$ holds.
  We now prove that every vertex $(x,y)$ of $Q$ not lying in $Q_1 \cap Q_2$
  lies in at most (thus, exactly) $d+1$ facets of $Q$.
  First, $y$ can satisfy at most $d$ inequalities of $Ay \leq b$ with equality
  because $P_1$ is simple.
  Second, $(x,y)$ can satisfy at most one of the other two inequalities
  with equality since otherwise, $\scalprod{a}{x} = \beta$ would hold,
  contradicting the fact that $(x,y) \notin Q_1 \cap Q_2$.
  Hence, the vertex lies in at most $d+1$ facets which proves the claim.
  This already proves that $Q$ is simple in the case $F = \emptyset$,
  since then there are no further vertices.

  It remains to show that if $F$ is a facet of $P_1$
  then every vertex $(x,y)$ of $Q_1 \cap Q_2$
  has at most $d+1$ neighbors in $Q$.
  In this case, $Q_1 \cap Q_2$ is a facet of $Q_1$ and of $Q_2$ which in turn are facets of $Q$.
  Since $Q_1 \cap Q_2$ is a facet of the simple polytope $Q_i$ for $i=1,2$,
  the vertex $(x,y)$ has $d-1$ neighbors in the (simple) facet $Q_1 \cap Q_2$
  and $1$ neighbor in $Q_i \setminus (Q_1 \cap Q_2)$.
  In total, $(x,y)$ has $d+1$ neighbors, because
  all vertices of $Q$ are vertices of $Q_1$ or $Q_2$
  since for fixed $y$ with $Ay \leq b$, any $x$ with
  $(x-y) \in \lin[a]$ must satisfy one of the other two inequalities
  with equality if it is an extreme point.
\end{proof}

An interesting observation is that in case of a reflection at a hyperplane $H_=$
which does not intersect the given polytope $P$, the resulting
extension polytope is combinatorially equivalent to $P \times [0,1]$.
This yields a (deformed) cube if such a reflection is applied iteratively
if the initial polytope is a cube.
Examples are the extensions of size $2\log m$ for regular $m$-gons 
for the case of $m = 2^k$ with $k \in \N$.

\begin{theorem}
  \label{TheoremRegularPolygons}
  Let $k \in \N$.
  The simple extension complexity of a regular $2^k$-gon is at most $2k$.
\end{theorem}

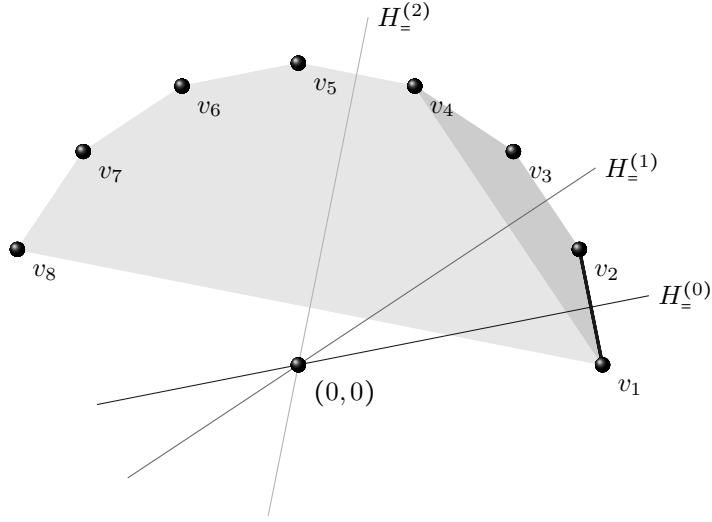
\begin{figure}[ht]
  \begin{center}
    \begin{tikzpicture}
      \piInputConfigured{TikZ-extended-formulations-m-gon.tex}[job=draw,colors,decorations]
    \end{tikzpicture}
  \end{center}
  \caption{Some Reflections used in the Proof of Theorem~\ref{TheoremRegularPolygons} for a $16$-gon.}
  \label{FigureReflections16gon}
\end{figure}

\begin{proof}
We recursively define a series of polytopes as follows:
The initial (simple) polytope is $P^{(0)} := \setdef{(1,0)}$, i.e., a single point.
Since the extensions we construct are located in increasingly higher-dimensional spaces
we write coordinates as $(x,y,z) \in \R \times \R \times \R^*$,
where the dimension of the $z$-space increases, initially being zero.

We now define for $i=0,1,2,\ldots,k-1$ the polytope $P^{(i+1)}$
as the reflection of $P^{(i)}$ at the halfspace 
$$H^{(i)}_{\leq} := \setdef{ (x,y,z) \in \R \times \R \times \R^* }[%
  -\sin((2^i-1) \cdot \pi/2^k )x + \cos((2^i-1) \cdot \pi/2^k)y \leq 0 ]. $$
Theorem~3 in \cite{KaibelP11a} shows that $P^{(k)}$ is an extension of a regular $2^k$-gon.
If we label the vertices of this $2^k$-gon with $v_1,v_2,\ldots,v_{2^k}$
in counter-clockwise order starting with
$v_1 = (1,0)$, the proof even shows that the projection of $P^{(i)}$ onto the first two coordinates
equals the convex hull of the vertices $v_1,v_2,\ldots,v_{2^i}$.
Now for every $i=0,1,\ldots,k-1$, the polytope $P^{(i)}$ does not intersect $H^{(i)}_=$
since the projection of such an intersection point would lie outside the mentioned convex hull.

This ensures that by induction all polytopes $P^{(i)}$ for $i=0,1,2,\ldots,k$ are simple
by Theorem~\ref{TheoremReflectionSimple}
and that the last polytope $P^{(k)}$ is a simple extension of the regular $2^k$-gon.
\end{proof}

\subsection{Disjunctive Programming}

The third major technique to construct extended formulations
is by means of \emph{disjunctive programming}, introduced by Balas \cite{Balas74},\cite{Balas79}.
We only consider the special case of a disjunction of two polytopes
$P_1, P_2 \subseteq \R^n$ and are interested in an extension
of the convex hull of the union of the two.

A helpful tool is the \emph{homogenization} $\homog P$ of a polytope $P$,
defined as $\homog P := \cone[P \times \setdef{1}]$,
where $\cone[\cdot]$ denotes the conic hull.
We say that a pointed polyhedral cone $C$ is \emph{weakly simple} if every
extreme ray of $C$ lies in exactly $\dim[C]-1$ facets
and \emph{strongly simple} if $C$ is a simple polyhedron.
Clearly, a strongly simple cone is also weakly simple.
Furthermore, if we have $C = \homog P$ then
$C$ is weakly simple if and only if $P$ is simple
and $C$ is strongly simple if and only if $P$ is a simplex.
We will need the following lemma about
weak simplicity of cartesian products of cones.

\begin{lemma}
  \label{TheoremSimpleCartesianProduct}
  Given two pointed polyhedral cones $C_1 \subseteq \R^{n_1}$,
  $C_2 \subseteq \R^{n_2}$,
  their product cone $C := C_1 \times C_2 \subseteq \R^{n_1+n_2}$
  is weakly simple if and only if both $C_1$ and $C_2$ are
  strongly simple.
\end{lemma}

\begin{proof}
  It is easy to check that
  $C_1 \times C_2 = \setdef{ (x_1,x_2) \in \R^{n_1+n_2} }[ x_i \in C_i ~~i=1,2]$
  is a pointed polyhedral cone again.
  Furthermore, the faces of $C_1 \times C_2$ are exactly the products
  of faces of $C_1$ and $C_2$, their dimensions add up,
  and a face $F_1 \times F_2$ of $C_1 \times C_2$ is contained in another face
  $G_1 \times G_2$ if and only if $F_1 \subseteq G_1$ and $F_2 \subseteq G_2$ hold.

  Hence, the extreme rays of $C_1 \times C_2$ are either
  products of extreme rays of $C_1$ with $\zerovec[n_2]$ or
  products of $\zerovec[n_1]$ with extreme rays of $C_2$.
  Similarly, the facets of $C_1 \times C_2$ are either
  products of facets of $C_1$ with $C_2$ or
  products of $C_1$ with facets of $C_2$.

  We consider an extreme ray of $C$,
  w.l.o.g. of the form $r \times \zerovec[n_2]$, where $r$ is an extreme ray $r$ of $C_1$.
  It is clearly contained in the facets $F_1 \times C_2$ where $F_1$ is a facet of $C_1$ containing $r$.
  Now for every facet $F_2$ of $C_2$, 
  we have $\zerovec[n_2] \subseteq F_2$ 
  and hence $C_1 \times F_2$ contains $r \times \zerovec[n_2]$.

  Thus, if $r$ is contained in $k$ facets of $C_1$ and $C_2$ has $\ell$ facets
  then $r \times \zerovec[n_2]$ is contained in $k + \ell$ facets of $C_1 \times C_2$.

  We always have $k \geq \dim C_1 - 1$ and $\ell \geq \dim C_2$ since $C_1, C_2$ are 
  pointed polyhedral cones.
  Hence, $k + \ell \geq \dim C_1 + \dim C_2 - 1$ holds and
  we have equality if and only if $k = \dim C_1 - 1$ and $\ell = \dim C_2$ are satisfied,
  and hence $C_1$ and $C_2$ are strongly simple.
\end{proof}

We now turn to the mentioned extension of $P = \conv[P_1 \cup P_2]$.
Define 
$$Q = \setdef{ (x^1,\lambda_1,x^2,\lambda_2) \in \homog[P_1] \times \homog[P_2] }%
[ \lambda_1 + \lambda_2 = 1 ] .$$
It is well-known that $Q$ together with the projection 
$(x^1,\lambda_1,x^2,\lambda_2) \mapsto x^1 + x^2$
yields an extension of $P$.
we now characterize when $Q$ is a simple polytope.

\begin{theorem}
  The extension polytope $Q$ of the disjunctive program for
  the polytope $P = \conv[P_1 \cup P_2]$
  is simple if and only if $P_1$ and $P_2$ are simplices.
\end{theorem}

\begin{proof}
  As $Q$ is the intersection of the pointed cone
  $C = \homog[P_1] \times \homog[P_2]$
  with the hyperplane defined by $\lambda_1 + \lambda_2 = 1$
  (which does not contain any of $C$'s extreme rays),
  we know that $Q$ is simple if and only if $C$ is weakly simple.
  Now Lemma~\ref{TheoremSimpleCartesianProduct} yields the result.
\end{proof}

\pagebreak[3]
\section{Bounding Techniques}
\label{SectionTechniques}

Let $P \subseteq \R^n$ be a polytope with $N$ vertices.
The faces of $P$ form a graded lattice $\polyLattice[P]$,
ordered by inclusion (see \cite{Ziegler01}).

Clearly, $P$ is the set of all convex combinations of its vertices,
immediately providing an extended formulation of size $N$:
\begin{equation*}
  P = \proj[x] \setdef{ (x,y) \in \R^n \times \R_+^V }[
    x = \sum_{v \in V} y_v v,~ \sum_{v \in V} y_v = 1]
\end{equation*}
Here, $\proj[x][\cdot]$ denotes the projection onto the space of $x$-variables
and $V$ is the set of vertices of $P$.
Note that this \emph{trivial extension} is simple since the extension is an $(N-1)$-simplex.

An easy observation for extensions $P = \pi(Q)$ with $Q \subseteq \R^d$ and $\pi : \R^d \to \R^n$ is that
the assignment $F \mapsto \pi^{-1}(F) \cap Q = \setdef{ y \in Q }[ \pi(y) \in F ]$ defines
a map $j$ which embeds $\polyLattice[P]$ into $\polyLattice[Q]$,
i.e., it is one-to-one and preserves inclusion in both directions (see \cite{FioriniKPT11}).
Note that this embedding furthermore satisfies
$j(F \cap F') = j(F) \cap j(F')$ for all faces $F,F'$ of $P$
(where the nontrivial inclusion $j(F) \cap j(F') \subseteq j(F \cap F')$
follows from $\pi(j(F) \cap j(F')) \subseteq \pi(j(F)) \cap \pi(j(F')) = F \cap F'$).
We use the shorthand notation $j(v) := j(\setdef{v})$ for vertices $v$ of $P$.

We consider the \emph{face-vertex non-incidence graph} $\BicliqueGraph[P]$
which is a bipartite graph having the faces and the vertices of $P$
as the node set and edges $\setdef{F,v}$ for all $v \notin F$.
Every facet $\hat{f}$ of an extension induces two node sets of this graph
in the following way:
\begin{gather}
  \begin{array}{rcl}
    \BicliqueFaces[\hat{f}] & := & 
      \setdef{ F \text{ face of } P }[ j(F) \subseteq \hat{f} ] \\
    \BicliqueVertices[\hat{f}] & := & 
      \setdef{ v \text{ vertex of } P }[ j(v) \not\subseteq \hat{f} ] 
  \end{array}
\end{gather}
We call $\BicliqueFaces[\hat{f}]$ and $\BicliqueVertices[\hat{f}]$ 
the \emph{set of faces} (resp. \emph{vertices}) \emph{induced by the facet} $\hat{f}$ 
(with respect to the extension $P = \pi(Q)$).
Typically, the extension and the facet $\hat{f}$ are fixed and
we just write $\BicliqueFaces$ (resp. $\BicliqueVertices$).
It may happen that $\BicliqueVertices[\hat{f}]$ is equal to the whole vertex set,
e.g., if $\hat{f}$ projects into the relative interior of $P$.
If $\BicliqueVertices[\hat{f}]$ is a proper subset of the vertex set
we call facet $\hat{f}$ \emph{proper} w.r.t. the projection.

For each facet $\hat{f}$ of an extension of $P$ 
the face and vertex sets $\BicliqueFaces[\hat{f}]$, $\BicliqueVertices[\hat{f}]$ together induce a biclique (i.e., complete bipartite subgraph)
in $\BicliqueGraph[P]$.
It follows from Yannakakis \cite{Yannakakis91} that every edge in $\BicliqueGraph[P]$
is covered by at least one of those induced bicliques.
We provide a brief combinatorial argument for this 
(in particular showing that we can restrict to proper facets)
in the proof of the following proposition.

\pagebreak[3]

\begin{proposition}
  \label{TheoremBicliqueCovering}
  Let $P = \pi(Q)$ be an extension.

  Then the subgraph of $\BicliqueGraph[P]$ induced by 
  $\BicliqueFaces[\hat{f}] \cupdot \BicliqueVertices[\hat{f}]$
  is a biclique for every facet $\hat{f}$ of $Q$.
  Furthermore, every edge $\setdef{F, v}$ of $\BicliqueGraph[P]$
  is covered by at least one of the bicliques induced
  by a proper facet.
\end{proposition}

\begin{proof}
  Let $\hat{f}$ be one of the facets and assume that an edge $\setdef{F,v}$
  with $F \in \BicliqueFaces[\hat{f}]$ and $v \in \BicliqueVertices[\hat{f}]$
  is not present in $\BicliqueGraph[P]$, i.e., $v \in F$.
  From $v \in F$ we obtain $j(v) \subseteq j(F) \subseteq \hat{f}$,
  a contradiction to $v \in \BicliqueVertices[\hat{f}]$.

  To prove the second statement, let $\setdef{F,v}$ be any edge
  of $\BicliqueGraph[P]$, i.e., $v \notin F$.
  Observe that the preimages $G := j(F)$ and $g := j(v)$
  are also not incident since $j$ is a lattice embedding.
  As $G$ is the intersection of all facets of $Q$ it is contained in
  (the face-lattice of a polytope is coatomic),
  there must be at least one facet $\hat{f}$ containing $G$ but not
  $g$ (since otherwise $g$ would be contained in $G$), yielding
  $F \in \BicliqueFaces[\hat{f}]$ and $v \in \BicliqueVertices[\hat{f}]$.

  If $F \neq \emptyset$, any vertex $w \in F$ satisfies 
  $j(w) \subseteq G \subseteq \hat{f}$ and hence $\hat{f}$ is a proper facet.
  If $F = \emptyset$, let $w$ be any vertex of $P$ distinct from $v$.
  The preimages $j(v)$ and $j(w)$ clearly satisfy
  $j(v) \not\subseteq j(w)$. 
  Again, since the face-lattice of $Q$ is coatomic,
  there exists a facet $\hat{f}$ with $j(w) \subseteq \hat{f}$
  but $j(v) \not\subseteq \hat{f}$.
  Hence, $\hat{f}$ is a proper facet and (since $\emptyset = F \subseteq \hat{f}$)
  $F \in \BicliqueFaces[\hat{f}]$ and $v \in \BicliqueVertices[\hat{f}]$ holds.
\end{proof}

Before moving on to simple extensions we mention two useful properties
of the induced sets. 
Both can be easily verified by examining the definitions of $\BicliqueFaces$
and $\BicliqueVertices$. 
See Figure~\ref{FigureBicliquesFaceLattice} for an illustration.

\begin{lemma}
  \label{TheoremBicliqueVerticesAndSubfaces}
  Let $\BicliqueFaces$ and $\BicliqueVertices$ be the face and vertex sets 
  induced by a facet of an extension of $P$, respectively.
  Then $\BicliqueFaces$ is closed under taking subfaces
  and $\BicliqueVertices = \setdef{ v \text{ vertex of } P }[ v \notin \bigcup_{F \in \BicliqueFaces} F ]$ holds.
\end{lemma}

\begin{figure}[ht]
  \begin{center}
    \begin{tikzpicture}
      \piInputConfigured{TikZ-extended-formulations-bicliques.tex}%
        [job=face-vertex-biclique,\tikzOptions]
    \end{tikzpicture}
  \end{center}
  \caption{The Sets $\BicliqueFaces$ and $\BicliqueVertices$ in the Face Lattice.}
  \label{FigureBicliquesFaceLattice}
\end{figure}

For the remainder of this section we assume that the extension polytope $Q$
is a \emph{simple} polytope and that $\BicliqueFaces$ and $\BicliqueVertices$
are face and vertex sets induced by a facet of $Q$.

\begin{theorem}
  \label{TheoremSimpleBicliqueFaces}
  Let $\BicliqueFaces$ and $\BicliqueVertices$ be the face and vertex sets 
  induced by a facet of a \emph{simple} extension of $P$, respectively.
  Then
  \begin{enumerate}[(a)]
  \item
    all pairs $(F,F')$ of faces of $P$ with $F \cap F' \neq \emptyset$
    and $F,F' \notin \BicliqueFaces$
    satisfy $F \cap F' \notin \BicliqueFaces$,
  \item
    the (inclusion-wise) maximal elements in $\BicliqueFaces$ are facets of $P$,
  \item
    and every vertex $v \notin \BicliqueVertices$ is contained in
    some facet $F$ of $P$ with $F \in \BicliqueFaces$.
  \end{enumerate}
\end{theorem}

\begin{proof}
  Let $\hat{f}$ be the facet of $Q$ inducing $\BicliqueFaces$ and $\BicliqueVertices$
  and $F,F'$ two faces of $P$ with non-empty intersection.
  Since $F \cap F' \neq \emptyset$, we have $j(F \cap F') \neq \emptyset$,
  thus the interval in $\polyLattice[Q]$ between $j(F \cap F')$
  and $Q$ is a Boolean lattice, i.e., isomorphic to the face-lattice of a simplex,
  (because $Q$ is simple, see Proposition~2.16~in~\cite{Ziegler01}).
  Suppose $F \cap F' \in \BicliqueFaces[\hat{f}]$.
  Then $\hat{f}$ is contained in that interval and it is a coatom,
  hence it contains at least one of $j(F)$ and $j(F')$ due to $j(F) \cap j(F') = j(F \cap F')$.
  But this implies $j(F) \in \BicliqueFaces$ or $j(F') \in \BicliqueFaces$,
  proving (a).

  For (b), let $F$ be an inclusion-wise maximal face in $\BicliqueFaces$ but not a facet of $P$.
  Then $F$ is the intersection of two faces $F_1$ and $F_2$ of $P$ properly containing $F$.
  Due to the maximality of $F$, $F_1, F_2 \notin \BicliqueFaces$
  but $F_1 \cap F_2 \in \BicliqueFaces$, contradicting (a).

  Statement (c) follows directly from (b) and Lemma~\ref{TheoremBicliqueVerticesAndSubfaces}.
\end{proof}

In order to use the Theorem~\ref{TheoremSimpleBicliqueFaces} for deriving lower bounds on
the sizes of simple extensions of a polytope $P$, 
one needs to have good knowledge of parts of the face lattice of $P$.
The part one usually knows most about is formed by the vertices and edges of $P$.
Therefore, we specialize Theorem~\ref{TheoremSimpleBicliqueFaces} to these faces for later use.

Let $G = (V,E)$ be a graph and denote by $\delta(W) \subseteq E$ the cut-set of
a node-set $W$.
Define the \emph{common neighbor operator} $\CommonNeighbors[\cdot]$ by
\begin{equation}
  \label{EquationCommonNeighbors}
  \CommonNeighbors[W] := 
    W \cup \setdef{v \in V}[\exists \setdef{u,v},\setdef{v,w} \in \delta(W) : u \neq w] \ .
\end{equation}
A set $W \subseteq V$ is then a (proper) \emph{common neighbor closed}
(for short \emph{$\CommonNeighbors$-closed}) set if
$\CommonNeighbors[W] = W$ (and $W \neq V$) holds.
We call sets $W$ with a minimum node distance of at least $3$
(i.e., the distance-$2$-neighborhood of a node $w \in W$ 
 does not contain another node $w' \in W$)
\emph{isolated}.
Isolated node sets are clearly $\CommonNeighbors$-closed.
Note that singleton sets are isolated and hence proper $\CommonNeighbors$-closed.
In particular, the vertex sets induced by the facets of the
trivial extension (see beginning of Section~\ref{SectionTechniques})
are the singleton sets.

Using this notion, we obtain the following corollary of Theorem~\ref{TheoremSimpleBicliqueFaces}.

\begin{corollary}
  \label{TheoremSimpleBicliqueVertices}
  The vertex set $\BicliqueVertices$
  induced by a proper facet of a \emph{simple} extension of $P$
  is a proper $\CommonNeighbors$-closed set.
\end{corollary}

\begin{proof}
  Theorem~\ref{TheoremSimpleBicliqueFaces} implies that
  for every $\setdef{u,v},\setdef{v,w}$ of (distinct) adjacent edges of $P$,
  we have
  \begin{gather*}
    \setdef{u,v},\setdef{v,w} \notin \BicliqueFaces 
      ~\Rightarrow~ \setdef{v} \notin \BicliqueFaces \ .
  \end{gather*}
  Due to Lemma~\ref{TheoremBicliqueVerticesAndSubfaces},
  $\BicliqueVertices = \setdef{ v \text{ vertex of } P }[ v \notin \bigcup \BicliqueFaces ]$,
  where $\BicliqueFaces$ is the face set induced by the same facet.
  Hence, $v \notin \BicliqueVertices$ implies 
  $\setdef{u,v} \in \BicliqueFaces$ or $\setdef{v,w} \in \BicliqueFaces$,
  thus $u \notin \BicliqueVertices$ or 
  $w \notin \BicliqueVertices$ and we conclude that $\BicliqueVertices$
  is $\CommonNeighbors$-closed.

  Furthermore, $\BicliqueVertices$ is not equal to the whole vertex set of $P$
  since the given facet is proper.
\end{proof}

We just proved that every biclique $\BicliqueFaces \cupdot \BicliqueVertices$
induced by a (proper) facet from a simple extension
must satisfy certain properties.
The next example shows that these properties are not sufficient for
an extension polytope to be simple.

\begin{example}
  \label{ExampleNonSimpleExtensionWithSimpleBicliqueCovering}
  Define $m_1, \ldots, m_7 \in \R^3$ to be the columns of the matrix
  \begin{equation*}
    M := \begin{pmatrix}
       1 &  5 &  1 &  0 & -1 & -5 & -1 \\
       1 &  0 & -1 &  0 &  1 &  0 & -1 \\
       0 & -4 &  0 &  1 &  0 & -4 &  0
    \end{pmatrix} \ ,
  \end{equation*}
  and let $Q := \conv \setdef{m_1, \ldots, m_7} \subseteq \R^3$ be their convex hull.
  The vertex $m_4$ has $4$ neighbors, that is, $Q$ is not simple.
  Let $P$ be the projection of $Q$ onto the first two coordinates.
  Observe that $P$ is a $6$-gon and that the only relevant
  types of faces $F,F'$ are adjacent edges of $P$.
  It is quickly verified that all
  induced face and vertex sets
  satisfy Theorem~\ref{TheoremSimpleBicliqueFaces}
  and Corollary~\ref{TheoremSimpleBicliqueVertices}, respectively.
\end{example}

\begin{figure}[ht]
  \begin{center}
    \begin{tikzpicture}
      \piInputConfigured{TikZ-extended-formulations-simple-extensions.tex}%
        [job=example-polytope,\tikzOptions]
    \end{tikzpicture}
  \end{center}
  \caption{Polytope $Q$ from Example~\ref{ExampleNonSimpleExtensionWithSimpleBicliqueCovering}
    and its projection $P$.}
  \label{FigureNonSimpleExtensionPolytope}
\end{figure}
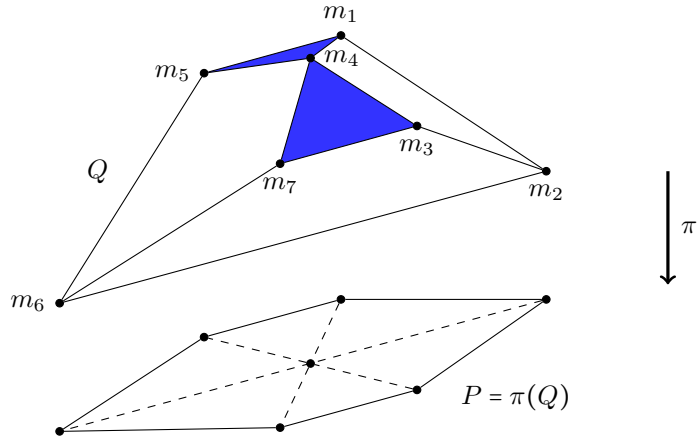

Note that this example only shows that we cannot decide from the biclique covering
whether the extension is simple.
It may still be true that for such biclique coverings there always
\emph{exists} a simple extension.

The polytope $Q$ from the example
can be used to show that Corollary~\ref{TheoremSimpleBicliqueVertices}
is indeed a specialization of Theorem~\ref{TheoremSimpleBicliqueFaces}~(a).
To see this, consider the set $\BicliqueFaces$ of faces
consisting of $\conv \setdef{m_1, m_4, m_5}$, $\conv \setdef{m_3, m_4, m_7}$
and all their subfaces.
Lemma~\ref{TheoremBicliqueVerticesAndSubfaces} implies
$\BicliqueVertices = \setdef{m_2, m_6}$ which is proper $\CommonNeighbors$-closed.
But $\BicliqueFaces$ does not satisfy Theorem~\ref{TheoremSimpleBicliqueFaces}~(a)
for the choice $F := \conv \setdef{m_1, m_2, m_3, m_4} \notin \BicliqueFaces$,
$F' := \conv \setdef{m_4,m_5,m_6,m_7} \notin \BicliqueFaces$
since $F \cap F' = \setdef{m_4} \in \BicliqueFaces$.

\bigskip
\pagebreak[3]

Nevertheless we can obtain useful lower bounds 
from Theorem~\ref{TheoremSimpleBicliqueFaces} and
Corollary~\ref{TheoremSimpleBicliqueVertices}.

\begin{corollary}
  \label{TheoremSimpleBicliqueVerticesCoveredByClosed}
  The node set of a polytope $P$ can be covered by $\sxc[P]$
  many proper $\CommonNeighbors$-closed sets.
\end{corollary}

\begin{lemma}
  \label{TheoremTrivialClosureImpliesDegreeBound}
  Let $P$ be a polytope and $G$ its graph.
  If all proper $\CommonNeighbors$-closed sets in $G$ are isolated
  then the simple extension complexity of $P$ is greater than the maximum
  size of the neighborhood of any node of $G$.
\end{lemma}

\begin{proof}
  Let $w$ be a node maximizing the size of the neighborhood
  and let $W$ be the neighborhood of $w$.
  Since no isolated set can contain more than one node from $W \cup \setdef{w}$,
  Corollary~\ref{TheoremSimpleBicliqueVerticesCoveredByClosed} implies the claim.
\end{proof}

Using knowledge about random 0/1-polytopes, we can easily establish the following result.

\begin{theorem}
  \label{TheoremRandomPolytopesSimpleExtensionComplexity}
  There is a constant $\sigma > 0$ such that a random $d$-dimensional 0/1-polytope $P$
  with at most $2^{\sigma d}$ vertices
  asymptotically almost surely has a simple extension complexity 
  equal to its number of vertices.
\end{theorem}

\begin{proof}
  One of the main results of Gillmann’s thesis (See Theorem~3.37 in \cite{Gillmann07} for $k=2$) is
  that there is such a $\sigma$ ensuring that a random $d$-dimensional 0/1-polytope $P$
  with at most $2^{\sigma d}$ vertices 
  asymptotically almost surely has every pair of vertices adjacent.
  Since in this situation the only proper $\CommonNeighbors$-closed sets are
  the singletons, Corollary~\ref{TheoremSimpleBicliqueVerticesCoveredByClosed} yields the claim.
\end{proof}

\pagebreak[3]
\section{\texorpdfstring{$k$}{k}-Hypersimplex}
\label{SectionHypersimplex}

Let $\Delta(k)$ denote the
$k$-hypersimplex in $\R^n$, i.e.,
the 0/1-cube intersected with
the hyperplane $\scalprod{\onevec[n]}{x} = k$.
Note that its vertices are all 0/1-vectors with exactly $k$ $1$'s,
since the above linear system is totally unimodular
(a row of ones together with two unit matrices).
It follows from the knowledge about edges and 2-faces of the cube that two vertices
of $\Delta(k)$ are adjacent if and only if they differ in exactly two coordinates.
In other words, all neighbors of a vertex $x$ can be obtained by
replacing a $1$ by a $0$ at some index and a $0$ by a $1$ at some other index.
Observe that $\Delta(k)$ is almost simple for $2 \leq k \leq n-2$ in
the sense that its dimension is $n-1$,
but every vertex lies in exactly $n$ facets.
With this in mind, the following result
may seem somewhat surprising.

\begin{theorem}
  \label{TheoremHypersimplex}
  Let $1 \leq k \leq n-1$.
  The simple extension complexity of
  $\Delta(k) \subseteq \R^n$ is equal
  to its number of vertices
  $\bincoeff{n}{k}$.
\end{theorem}

\begin{proof}
  The case of $k=1$ or $k=n-1$ is clear since then
  $\Delta(k)$ is an $(n-1)$-dimensional simplex.

  Let $2 \leq k \leq n-2$ and
  $\BicliqueFaces$ and $\BicliqueVertices$
  be face and vertex sets induced by a proper facet of
  a simple extension of $\Delta(k)$.

  Since every vertex $v$ of $\Delta(k)$ has
  $v_i = 0$ or $v_i = 1$, at most one of the
  facets $x_i \geq 0$ or $x_i \leq 1$ can be
  in $\BicliqueFaces$ for every $i \in [n]$ (otherwise $\BicliqueVertices$ would be empty).
  We can partition $[n]$ into $L \cupdot U \cupdot R$
  such that $L$ (resp. $U$) contains those
  indices $i \in [n]$ such that the facet
  corresponding to $x_i \geq 0$ (resp. $x_i \leq 1$) is in $\BicliqueFaces$
  and $R$ contains the remaining indices.
  Lemma~\ref{TheoremBicliqueVerticesAndSubfaces}
  yields
  \begin{equation}
    \BicliqueVertices = \setdef{ v \text{ vertex of } \Delta(k) }[
      v_L = \onevec,~ v_U = \zerovec ] \ .
  \end{equation}
  We now prove that a node set $\BicliqueVertices$ of this form
  is proper $\CommonNeighbors$-closed only if $|\BicliqueVertices| = 1$.
  Then, Corollary~\ref{TheoremSimpleBicliqueVerticesCoveredByClosed} yields the claim.

\begin{figure}[ht]
  \begin{center}
    \begin{tikzpicture}
      \piInputConfigured{TikZ-extended-formulations-simple-hypersimplex.tex}%
        [job=hypersimplex,\tikzOptions]
    \end{tikzpicture}
  \end{center}
  \caption{Vertices of $\Delta(k)$ in $\BicliqueVertices$ for a Biclique.}
  \label{FigureHypersimplex}
\end{figure}

  Indeed, if we have $|\BicliqueVertices| > 1$, then there exist
  vertices $u,w \in \BicliqueVertices$ and indices $\khsi,\khsj \in R$ such that
  $u_{\khsi} = w_{\khsj} = 1$, $u_{\khsj} = w_{\khsi} = 0$, and $u_l = w_l$
  for all $l \notin \{\khsi,\khsj\}$ (see Figure~\ref{FigureHypersimplex}).
  Choose any $s \in L \cupdot U$ and observe that, since $u,w \in \BicliqueVertices$,
  $u_s = w_s = 1$ if $s \in L$ and $u_s = w_s = 0$ if $s \in U$.
  The following vertex is easily checked to be adjacent 
  to $u$ and $w$ ($\min$ and $\max$ must be read component-wise):
  \begin{equation*}
    \khsv := \left\{\begin{array}{ll}
      \max(u,w) - \unitvec{s} & \text{ if } s \in L \\
      \min(u,w) + \unitvec{s} & \text{ if } s \in U
    \end{array}\right.
  \end{equation*}
  As $\khsv_s = 0$ if $s \in L$ and $\khsv_s = 1$ if $s \in U$,
  $\khsv \notin \BicliqueVertices$.
  This contradicts the fact that $\BicliqueVertices$ is $\CommonNeighbors$-closed.
\end{proof}

\pagebreak[3]
\section{Spanning Tree Polytope}
\label{SectionSpanningTreePolytope}

In this section we bound the simple extension complexity
of the spanning tree polytope
$\SpanTreePoly{K_n}$ of the complete graph $K_n$ with $n$ nodes.
In order to highlight different perspectives
we mention three equivalent adjacency characterizations
which all follow from the fact that the spanning tree polytope is the
base polytope of a graphic matroid (see \cite{Schrijver03}, Theorem~40.6.).
The vertices corresponding to spanning trees $T$ and $T'$ are adjacent in the spanning tree polytope if and only if \dots
\begin{itemize}
\item
  \dots $|T \Delta T'| = 2$ holds.
\item
  \dots $T'$ arises from $T$ by removing one edge and reconnecting the two connected components by another edge.
\item
  \dots $T'$ arises from $T$ by adding one additional adge and removing any edge from the cycle that this edge created.
\end{itemize}
From the third statement it is easy to see that the maximum degree of the $1$-skeleton of $\SpanTreePoly{K_n}$ is
in $\orderO{n^3}$, since there are $\orderO{n^2}$ possible choices for the additional edge, each of which
yields $\orderO{n}$ choices for a cycle-edge to remove.

\begin{lemma}
  \label{TheoremSpanningTreeCommonNeighborClosure}
  All proper $\CommonNeighbors$-closed sets in the graph of $\SpanTreePoly{K_n}$ are isolated.
\end{lemma}

\begin{proof}
  Throughout the proof, we will identify vertices with the corresponding spanning
  trees.

  Suppose $\BicliqueVertices$ is a proper $\CommonNeighbors$-closed set that is not isolated.
  Then there are spanning trees $T_1, T_2 \in \BicliqueVertices$ and 
  $T_3 \notin \BicliqueVertices$,
  such that $T_1$ is adjacent to both $T_2$ and $T_3$, but $T_2$ and $T_3$ are not adjacent.

  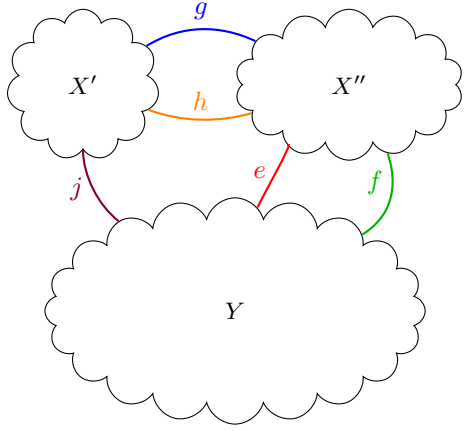
\begin{figure}[ht]
    \begin{center}
      \begin{tikzpicture}
        \piInputConfigured{TikZ-extended-formulations-simple-spanning-tree.tex}%
          [job=common-neighbors,\tikzOptions]
      \end{tikzpicture}
    \end{center}
    \caption{Case 2 of Lemma~\ref{TheoremSpanningTreeCommonNeighborClosure}.}
    \label{FigureSTcase2}
  \end{figure}

  Let $\sptTe$ be the unique edge which is in $T_1$ but not in $T_2$, i.e., $\setdef{\sptTe} = T_1 \setminus T_2$.
  Analogously, let $\setdef{\sptTf} = T_2 \setminus T_1$, $\setdef{\sptTg} = T_1 \setminus T_3$, and $\setdef{\sptTh} = T_3 \setminus T_1$.
  Since $T_2$ and $T_3$ are not adjacent in the polytope, their symmetric difference $T_2 \Delta T_3 \subseteq \setdef{\sptTe,\sptTf,\sptTg,\sptTh}$ must have cardinality greater than $2$.
  Because the symmetric difference of two spanning trees consists of an even number of edges, that cardinality must be equal to $4$, proving $\sptTe \neq \sptTg$.
  Let us define $F$ by $F = T_1 \setminus \setdef{\sptTe,\sptTg}$ which is a tree with two edges missing, i.e., a forest with three connected components $X', X'', Y$.
  W.l.o.g., $g$ connects $X'$ with $X''$ and $e$ connects $X''$ with $Y$.
  In their turn, $T_2$ and $T_3$ can be written as $F \cup \setdef{\sptTf,\sptTg}$ and $F \cup \setdef{\sptTe,\sptTh}$, respectively.

  There are two possible cases for $\sptTh$:

  \medskip

  \noindent
  \textbf{Case 1:} $\sptTh$ connects $Y$ with $X'$ or $X''$.

  Let $T' := F \cup \setdef{\sptTg,\sptTh}$ and observe that $T'$ is a spanning tree
  since $\sptTg$ connects $X'$ with $X''$ and $h$ connects one of both with $Y$.
  Obviously, $T'$ is adjacent to $T_1$, $T_2$, and $T_3$.
  Since $T'$ is adjacent to $T_1$ and $T_2$, 
  $T' \in \CommonNeighbors[\BicliqueVertices] = \BicliqueVertices$.
  Since $T_3$ is adjacent to $T_1, T' \in \BicliqueVertices$,
  this in turn implies the contradiction $T_3 \in \BicliqueVertices$.
 
  \medskip

  \noindent
  \textbf{Case 2:} $\sptTh$ connects $X'$ with $X''$.

  Let $\sptTj$ be any edge connecting $X'$ with $Y$ 
  (recall that we dealing with a complete graph)
  and let $T' := F \cup \setdef{\sptTg,\sptTj}$ which is a spanning tree adjacent
  to $T_1$ and $T_2$ and hence $T' \in \CommonNeighbors[\BicliqueVertices] = \BicliqueVertices$.
  Clearly, $T'' := F \cup \setdef{\sptTe,\sptTj}$
  is a spanning tree adjacent to $T_1$ and $T'$ and hence $T'' \in \BicliqueVertices$.
  Finally, let $T''' := F \cup \setdef{\sptTh,\sptTj}$ be a third spanning tree
  adjacent to $T'$ and $T''$. 
  Again, we have $T''' \in \BicliqueVertices$ due to 
  $\CommonNeighbors[\BicliqueVertices] = \BicliqueVertices$.

  Since $T_3$ is adjacent to $T_1$ and $T'''$,
  exploiting $\CommonNeighbors[\BicliqueVertices] = \BicliqueVertices$
  once more yields the contradiction $T_3 \in \BicliqueVertices$.
\end{proof}

Using this result we immediately get a lower bound of $\orderOmega{n^3}$
for the simple extension complexity of $\SpanTreePoly{K_n}$
since the maximum degree of its graph is of that order.
However, we can prove a much stronger result.

\begin{theorem}
  \label{TheoremSpanningTreePolytopeSimpleExtensionComplexity}
  The simple extension complexity of the spanning tree polytope of $K_n$
  is in $\orderOmega{2^{n-\ordero{n}}}$.
\end{theorem}

  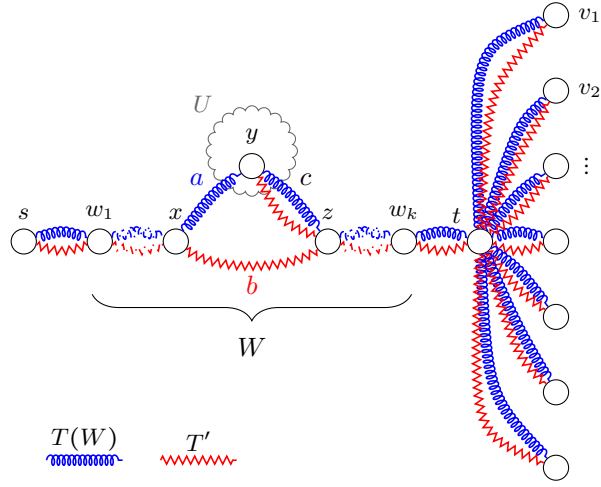
\begin{figure}[ht]
    \begin{center}
      \begin{tikzpicture}
        \piInputConfigured{TikZ-extended-formulations-simple-spanning-tree.tex}%
          [job=simple-bicliques,\tikzOptions]
      \end{tikzpicture}
    \end{center}
    \caption{Construction for Theorem~\ref{TheoremSpanningTreePolytopeSimpleExtensionComplexity}.}
    \label{FigureSPTpathStar}
  \end{figure}

\begin{proof}
  Assume $n \geq 5$ and
  let $s,t$ be any two distinct nodes of $K_n$.
  Consider the following set of subsets of the nodes $V \setminus \setdef{s,t}$
  \begin{equation*}
    \sptAllW{} := \setdef{ W \subseteq V \setminus \setdef{s,t} }[ |W| = \lfloor n/2 \rfloor ] \ .
  \end{equation*}
  Let $k := \lfloor n/2 \rfloor$, fix some ordering of the nodes
  $w_1, w_2, \ldots, w_k \in W$ for each $W \in \sptAllW$
  and define a specific tree $\sptTW$
  \begin{equation*}
    \begin{array}{rcl}
      \sptTW  & :=  & \setdef{ \setdef{s,w_1}, \setdef{w_k,t} } \\
              &     & \cup \setdef{ \setdef{w_i, w_{i+1}} }[ i \in [k-1] ] \\
              &     & \cup \setdef{ \setdef{t,v} }[ v \notin \left(W \cup \setdef{s,t}\right) ]
    \end{array}
  \end{equation*}
  as depicted in Figure~\ref{FigureSPTpathStar}.
  We will now prove that for each simple extension of $\SpanTreePoly{K_n}$ 
  every such $\sptTW$ must be in a different
  induced vertex set.

  Let $W \in \sptAllW$ be some set $W$ with tree $\sptTW$.
  Let $\BicliqueFaces$ and $\BicliqueVertices$ be the face and vertex sets,
  respectively, induced by a proper facet of a simple extension
  such that $\sptTW$ is in $\BicliqueVertices$.
  Construct an adjacent tree $\sptTprime$ as follows.
  
  Choose some vertex $y \in W$ and let $x$-$y$-$z$ be a subpath of the
  $s$-$t$-path in $\sptTW$ in that order.
  Note that $\setdef{x,y,z} \subseteq W \cup \setdef{s,t}$.
  Denote by $a$, $b$, $c$ the edges $\setdef{x,y}$, $\setdef{x,z}$, and $\setdef{y,z}$, respectively.

  Let $\sptTprime = \sptTW \setminus \setdef{a} \cup \setdef{b}$.
  Because $\sptTprime$ is adjacent to $\sptTW$,
  by Lemma~\ref{TheoremSpanningTreeCommonNeighborClosure} we know $\sptTprime \notin \BicliqueVertices$.
  Hence, due to Lemma~\ref{TheoremBicliqueVerticesAndSubfaces},
  there must be a facet $F \in \BicliqueFaces$ defined by $x(E[U]) \leq |U| - 1$
  (with $|U| \geq 2$) which contains $\sptTprime$.
  Furthermore, this facet does not contain $\sptTW$ because $\sptTW \in \BicliqueVertices$ holds.
  Hence, we have $|\sptTW[U]| < |U| - 1$ and $|\sptTprime[U]| = |U| - 1$.
  This implies $|\sptTW \cap \delta(U)| \geq 2$ and $|\sptTprime \cap \delta(U)| = 1$.
  Obviously, $a \in \delta(U)$ and $b \notin \delta(U)$.

  Then $x,z \in U$ if and only if $y \notin U$ because $a \in \delta(U)$ and $b \notin \delta(U)$.
  Hence, $c \in \delta(U)$, i.e., $T \cap \delta(U) = \setdef{c}$.
  Due to $|U| \geq 2$, this implies $U = V \setminus \setdef{y}$.

  As this can be argued for any $y \in W$, we have that the facets defined by
  $V \setminus \setdef{y}$ are in $\BicliqueFaces$ for all $y \in W$.
  Hence, $\BicliqueVertices$ contains only trees $T$
  for which $|T \cap \delta(V \setminus \setdef{y})| = |T \cap \delta(\setdef{y})| \geq 2$, i.e.,
  no leaf of $T$ is in $W$.

  This shows that for distinct sets $W,W' \in \sptAllW$,
  any vertex set $\BicliqueVertices$ 
  induced by a proper facet of a simple extension that contains $\sptTW$
  does not contain $T(W')$
  because any vertex $v \in W \setminus W'$ is a leaf of $T(W')$.
  Hence, the number of simple bicliques is at least 
  \begin{equation*}
    |\sptAllW| = \bincoeff{n-2}{\lfloor n/2 \rfloor} \in \orderOmega{2^{n - \ordero{n}}} \ .
  \end{equation*} 
\end{proof}

\pagebreak[4]
\section{Flow Polytopes for Acyclic Networks}
\label{SectionFlowPolytopes}

Many extended formulations model the solutions to the original formulation
via a path in a specifically constructed directed acyclic graph.
A simple example is the linear-size formulation for the parity polytope by Carr and Konjevod \cite{CarrK05},
and a more elaborate one is the approximate formulation for 0/1-knapsack polytopes by Bienstock \cite{Bienstock08}.

Let $D = (V,A)$ be a directed acyclic graph with fixed source $s \in V$ and sink $t \in V$.
By $\FlowPaths[s,t]{D}$ we denote the arc-sets of $s$-$t$-paths in $D$.
For some path $P \in \FlowPaths[s,t]{D}$ and nodes $u,v \in V(P)$,
we denote by $P|_{(u,v)}$ the subpath of $P$ going from $u$ to $v$.

For acyclic graphs, the convex hull of the characteristic vectors of all $s$-$t$-paths
is equal to the uncapacitated $s$-$t$-flow polytope $\FlowPoly[s-t]{D}$ with flow-value $1$,
since the linear description of the latter is totally unimodular.
The inequalities in this description correspond to nonnegativity constraints
of the arc variables, and a vertex corresponding to the path $P$ is obviously non-incident
to a facet corresponding to $y_a \geq 0$ if and only if $a \in P$ holds.
Adjacency in the path polytope was characterized by Gallo and Sodini \cite{GalloS78} and
can be stated as follows:
Two $s$-$t$-paths $P,P'$ correspond to adjacent vertices 
of the polytope if and only if
their symmetric difference consists of two paths from 
$x$ to $y$ ($x,y \in V$, $x \neq y$) without 
common inner nodes.
In other words, they must split and merge exactly once.

Such a network formulation can be easily decomposed into two
independent formulations if a node $v$ exists such that
every $s$-$t$-path traverses $v$.
We are now interested in the simple extension complexities
of flow polytopes of $s$-$t$-networks that cannot be decomposed
in such a trivial way.
Our main result in this section is the following:

\begin{theorem}
  \label{TheoremFlowSimpleExtensionComplexity}
  Let $D = (V,A)$ be a directed acyclic graph with source $s \in V$ and sink $t \in V$
  such that for every node $v \in V \setminus \setdef{s,t}$ there exists
  an $s$-$t$-path in $D$ which does not traverse $v$.

  Then the simple extension complexity of $\FlowPoly[s-t]{D} \subseteq \R_+^A$ is equal
  to the number of distinct $s$-$t$-paths $|\FlowPaths[s,t]{D}|$. 
\end{theorem}
  
\begin{proof}
  Let $\BicliqueFaces$ and $\BicliqueVertices$ be the face and vertex sets 
  induced by a proper facet of a simple extension of $\FlowPoly[s-t]{D}$, respectively.
  The goal is to prove $|\BicliqueVertices| = 1$, let us
  assume for the sake of contradiction $|\BicliqueVertices| \geq 2$.
  By Theorem~\ref{TheoremSimpleBicliqueFaces}~(b),
  the (inclusion-wise) maximal faces in $\BicliqueFaces$ are facets.
  Let $\emptyset \neq B' \subseteq A$ be the arc set corresponding to these facets.
  By Lemma~\ref{TheoremBicliqueVerticesAndSubfaces},
  $\BicliqueVertices$ is the set of (characteristic vectors of)
  paths $P \in \FlowPaths[s,t]{D}$ satisfying $P \supseteq B'$.
  Let $\flowB \subseteq A$ be the set of arcs common to all such paths
  and note that $\flowB \supseteq B' \neq \emptyset$.

  By construction, for any path 
  $P \in \BicliqueVertices$ and any arc $a \in P \setminus \flowB$, there is an alternative
  path $P' \in \BicliqueVertices$ with $a \notin P'$.

  Let us fix one of the paths $\flowP \in \BicliqueVertices$.
  Let, without loss of generality, $(x',x) \in \flowB$ be such that the
  arc of $\flowP$ leaving $x$ (exists and) is not in $\flowB$.
  If such an arc does not exist, since $\flowB \neq \flowP$,
  there must be an arc $(x,x') \in \flowB$ such that
  the arc of $\flowP$ entering $x$ is not in $\flowB$.
  In this case, revert the directions of all arcs in $D$
  and exchange the roles of $s$ and $t$ and apply subsequent arguments to the new network.
  Let $y$ be the first node on $\flowP|_{(x,t)}$ different 
  from $x$ and incident to some arc in $\flowB$ or,
  if no such $y$ exists, let $y := t$.
  Paths in $\BicliqueVertices$ must leave $x$ and enter $y$ but may differ inbetween.
  The set of traversed nodes is defined as
  $$S := \setdef{ v \in V \setminus \setdef{x,y} }[%
    \exists \,\textrm{$x$-$v$-$y$-path in } D ] \ .$$
  By construction, $x \notin \setdef{s,t}$ and by the assumptions of the Theorem
  there exists a path $\flowPprime \in \FlowPaths[s,t]{D}$
  which does not traverse $x$.
  Let $s'$ be the last node on $\flowP|_{(s,x)}$ that is traversed by $\flowPprime$.
  Analogously, let $t'$ be the first node of $V(\flowP|_{(x,t)}) \cup S$
  that is traversed by $\flowPprime$.
  Note that $t' \neq x$ since $t'$ is traversed by $\flowPprime$ but $x$ is not.
  We now distinguish two cases for which we show that $\BicliqueVertices$ 
  is not $\CommonNeighbors$-closed yielding a contradiction to 
  Corollary~\ref{TheoremSimpleBicliqueVertices}:
  
  \begin{figure}[ht]
    \begin{center}
      \begin{tikzpicture}
        \piInputConfigured{TikZ-extended-formulations-simple-acyclic-flows.tex}%
          [job=common neighbor,spindle,alternative path b,spindle path a,spindle path b,final paths,\tikzOptions]
      \end{tikzpicture}
    \end{center}
    \caption{Construction for Case~1 in the Proof of 
      Theorem~\ref{TheoremFlowSimpleExtensionComplexity}.}
    \label{FigureFlowSimpleExtensionComplexity1}
  \end{figure}
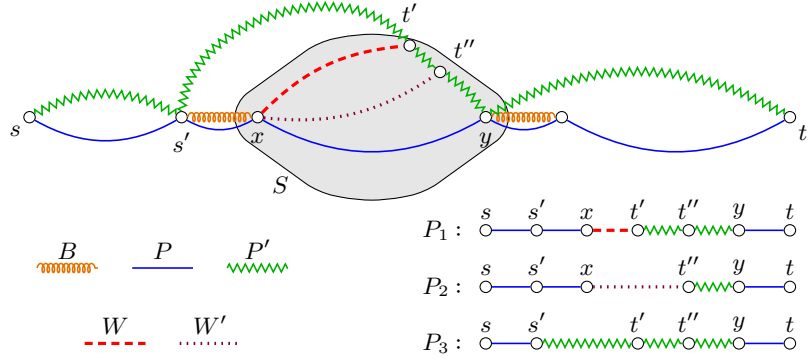
  
  \begin{figure}[ht]
    \begin{center}
      \begin{tikzpicture}
        \piInputConfigured{TikZ-extended-formulations-simple-acyclic-flows.tex}%
          [job=common neighbor,spindle,alternative path a,spindle path a,spindle path b,final paths,\tikzOptions]
      \end{tikzpicture}
    \end{center}
    \caption{Construction for Case~2 in the Proof of 
      Theorem~\ref{TheoremFlowSimpleExtensionComplexity}.}
    \label{FigureFlowSimpleExtensionComplexity2}
  \end{figure}
  
  \medskip
  
  \noindent
  \textbf{Case 1:} $t' \in S$.
  
  By definition of $S$ there must be an $x$-$t'$-$y$-path $\flowW$.
  Note that $t'$ could be equal to $y$ and then $\flowW$ could agree with $\flowP|_{(x,y)}$ as well.
  Let $(z,t') \in \flowW$ be the arc of $\flowW$ entering $t'$.
  By definition of $y$, we conclude that $(z,t') \notin \flowB$.
  Hence, there is an alternative $x$-$y$-path $\flowWprime \neq \flowW$ 
  which does not use $(z,t')$.
  We choose $\flowWprime$ such that it uses as many arcs of $\flowW|_{(t',y)}$ as possible.
  Construct the following three paths (see Figure~\ref{FigureFlowSimpleExtensionComplexity1}):
  \begin{equation*}
    \begin{array}{rcl}
      P_1 & := & \flowP|_{(s,x)} \cup \flowW \cup \flowP|_{(y,t)} \\
      P_2 & := & \flowP|_{(s,x)} \cup \flowWprime \cup \flowP|_{(y,t)} \\
      P_3 & := & \flowP|_{(s,s')} \cup \flowPprime|_{(s',t')} \cup \flowW|_{(t',y)} \cup \flowP|_{(y,t)}
    \end{array}
  \end{equation*}
  By construction $P_1, P_2 \in \BicliqueVertices$ but $P_3 \notin \BicliqueVertices$.
  $P_1$ and $P_3$ are adjacent in $\FlowPoly[s-t]{D}$ since they
  only differ in the disjoint paths from $s'$ to $t'$.
  Analogously, $P_2$ and $P_3$ are adjacent
  and thus, contradicting the fact that $\BicliqueVertices$ is $\CommonNeighbors$-closed.

  \medskip

  \noindent
  \textbf{Case 2:} $t' \notin S$.

  Let $\flowW := \flowP|_{(x,y)}$ and let $\flowWprime$ be 
  a different $x$-$y$-path which must exist by definition of $y$.
  Construct the following three paths (see Figure~\ref{FigureFlowSimpleExtensionComplexity2}):
  \begin{equation*}
    \begin{array}{rcl}
      P_1 & := & \flowP = \flowP|_{(s,x)} \cup \flowW \cup \flowP|_{(y,t)} \\
      P_2 & := & \flowP|_{(s,x)} \cup \flowWprime \cup \flowP|_{(y,t)} \\
      P_3 & := & \flowP|_{(s,s')} \cup \flowPprime|_{(s',t')} \cup \flowP|_{(t',t)}
    \end{array}
  \end{equation*}
  By construction $P_1, P_2 \in \BicliqueVertices$ but $P_3 \notin \BicliqueVertices$ 
  since it does not use $(x',x) \in \flowB$.
  $P_1$ and $P_3$ as well as $P_2$ and $P_3$ are adjacent in $\FlowPoly[s-t]{D}$ since they
  only differ in the disjoint paths from $s'$ to $t'$.
  Again, this contradicts the fact that $\BicliqueVertices$ is $\CommonNeighbors$-closed.
\end{proof}

\pagebreak[3]
\section{Perfect Matching Polytope}
\label{SectionPerfectMatchingPolytope}

The \emph{matching polytope} and the \emph{perfect matching polytope}
of a graph $G = (V,E)$ are defined as
\begin{align*}
  \MatchPoly{G}     &:= \conv \setdef{ \chi(M) }[ M \text{ matching in } G ] \\
  \PerfMatchPoly{G} &:= \conv \setdef{ \chi(M) }[ M \text{ perfect matching in } G ] \ ,
\end{align*}
where $\chi(M) \in \setdef{0,1}^E$ is the characteristic vector
of the set $M \subseteq E$, i.e., $\chi(M)_e = 1$ if and only if $e \in M$.
We mainly consider the (perfect) matching polytope of the complete graph
with $2n$ nodes $\PerfMatchPoly{K_{2n}}$.
Our main theorem here reads as follows:
\begin{theorem}
  \label{TheoremPerfectMatchingPolytopeSimpleExtensionComplexity}
  The simple extension complexity of the perfect matching polytope of $K_{2n}$
  is equal to its number of vertices $\frac{(2n)!}{n! \cdot 2^n}$.
\end{theorem}

We first give the high-level proof which uses a structural result presented afterwards.

\begin{proof}
The proof is based on Theorem~\ref{TheoremThreeCommonNeighbor}.
It states that for any three perfect matchings $M_1$, $M_2$, $M_3$ in $K_{2n}$,
where $M_1$ and $M_2$ are adjacent (i.e., the corresponding vertices are adjacent),
$M_3$ is adjacent to both $M_1$ and $M_2$ or there exists a 
fourth matching $M'$ adjacent to all three matchings.

Let $P = \PerfMatchPoly{K_{2n}}$
and suppose that $\BicliqueVertices$ is a proper $\CommonNeighbors$-closed set
with $|\BicliqueVertices| \geq 2$.
Since the polytope's graph is connected
there exists a matching $M_1 \notin \BicliqueVertices$ adjacent to 
some matching $M_2 \in \BicliqueVertices$.
Let $M_3 \in \BicliqueVertices \setminus \setdef{M_2}$.
As $\BicliqueVertices$ is $\CommonNeighbors$-closed 
and $M_3 \in \BicliqueVertices$ holds, $\setdef{M_1, M_2, M_3}$ cannot be a triangle.
Hence, by Theorem~\ref{TheoremThreeCommonNeighbor} mentioned above,
there exists a common neighbor matching $M'$.
Since $M'$ is adjacent to $M_2$ and $M_3$, we conclude $M' \in \BicliqueVertices$.
But now $M_1 \notin \BicliqueVertices$ is adjacent to the two matchings $M_2$ and $M'$
from $\BicliqueVertices$ contradicting the fact that $\BicliqueVertices$ is
$\CommonNeighbors$-closed.

Hence all proper $\CommonNeighbors$-closed sets are singletons
which implies the claim due to Corollary~\ref{TheoremSimpleBicliqueVerticesCoveredByClosed}.
\end{proof}

Since $\PerfMatchPoly{K_{2n}}$ is a face of $\MatchPoly{K_{2n}}$ and simple extensions
of polytopes induce simple extensions of their faces we obtain the following corollary
for the latter polytope.

\begin{corollary}
  \label{TheoremSimpleExtensionComplexityMatchingPolytope}
  The simple extension complexity of the matching polytope of $K_{2n}$
  is at least $\frac{(2n)!}{n! \cdot 2^n}$.
\end{corollary}

\subsection{Adjacency Result for the Perfect Matching Polytope}

We now turn to the mentioned result on the adjacency structure 
of the perfect matching polytope of $K_{2n}$. 
It is a generalization of the diameter result of Padberg and Rao's in \cite{PadbergR74}.

Clearly, the symmetric difference $M \Delta M'$ of two perfect 
matchings is always a disjoint union of alternating cycles, so-called \emph{$M$-$M'$-cycles}.
Chv{\'a}tal \cite{Chvatal75} showed that
(the vertices corresponding to) two perfect matchings $M$ and $M'$ are adjacent 
if and only if $M \Delta M'$ forms a single alternating cycle.
For an edge set $F$ we denote by $V(F)$ the set of nodes covered by the edges of $F$.
We start with an easy construction and modify the resulting matching later.

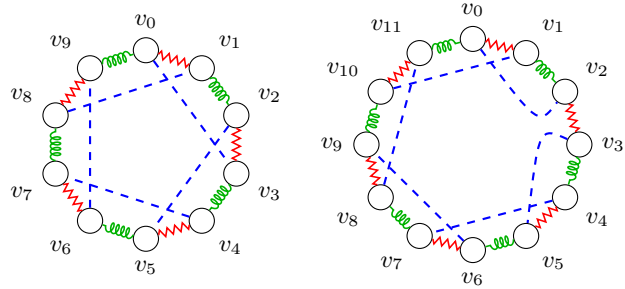
\begin{figure}[ht]
  \begin{center}
    \begin{tikzpicture}
      \piInputConfigured{TikZ-matchings-common-neighbor.tex}%
        [job=adjacent-cycle,\tikzOptions]
    \end{tikzpicture}
  \end{center}
  \caption{Lemma~\ref{TheoremAdjacentCycle} for a $10$-cycle and a $12$-cycle.}
  \label{FigureAdjacentCycle}
\end{figure}

\begin{lemma}
  \label{TheoremAdjacentCycle}
  For any adjacent perfect matchings $\pmMone$, $\pmMtwo$ there exists a perfect matching
  $\pmMprime$ adjacent to $\pmMone$ and $\pmMtwo$ that satisfies
  \begin{gather*}
    V(\pmMone \Delta \pmMtwo) = V(\pmMone \Delta \pmMprime) = V(\pmMtwo \Delta \pmMprime)
  \end{gather*}
  and $\pmMprime \cap (\pmMone \Delta \pmMtwo) = \emptyset$.
\end{lemma}

\begin{proof}
  Let $v_0, v_1, \ldots, v_{2l-1}, v_{2l} = v_0$
  be the set of ordered nodes of the cycle $M_1 \Delta M_2$
  and identify $v_{2l+1} = v_1$.
  If $l$ is odd,
  \begin{align*}
    \pmMprime &:= \setdef{\setdef{v_i, v_{i+3}}}[i = 0,2,4,6,\ldots,2l-2]
  \end{align*}
  induces $M_i$-$\pmMprime$-cycles visiting the nodes in the following order:
  \begin{gather*}
    \begin{array}{ll}
      \pmMone \Delta \pmMprime: & v_0, v_3, v_2, v_5, v_4, v_7, v_6, \ldots, v_{2l-1}, \\
                            & v_{2l-2}, v_1, v_0 \\
      \pmMtwo \Delta \pmMprime: & v_0, v_3, v_4, v_7,v_8, \ldots, v_{2l-3}, v_{2l-2}, \\
                            & v_1, v_2, v_5, v_6, \ldots, v_{2l-4}, v_{2l-1}, v_0
    \end{array}
  \end{gather*}
  If $l$ is even,
  \begin{align*}
      \pmMprime &:= \setdef{\setdef{v_i, v_{i+3}}}[i = 4,6,\ldots,2l-2] \\
                & \cupdot\, \setdef{ \setdef{v_0,v_2}, \setdef{v_3,v_5} }
  \end{align*}
  induces $M_i$-$\pmMprime$-cycles visiting the nodes in the following order:
  \begin{gather*}
    \begin{array}{ll}
      \pmMone \Delta \pmMprime: & v_0, v_2, v_3, v_5, v_4, v_7, v_6, \ldots, v_{2l-1}, \\
                            & v_{2l-2}, v_1, v_0 \\
      \pmMtwo \Delta \pmMprime: & v_0, v_2, v_1, v_{2l-2}, v_{2l-3}, \ldots, v_6, v_5, v_3, \\
                            & v_4, v_7, v_8, \ldots, v_{2l-4}, v_{2l-1}, v_0
    \end{array}
  \end{gather*}
  Figure~\ref{FigureAdjacentCycle} shows examples for both cases.
  It is easy to see that the node sets of the cycles equals the node set of $\pmMone \Delta \pmMtwo$
  and that $\pmMprime \cap (\pmMone \Delta \pmMtwo) = \emptyset$ holds.
  In order to produce a perfect matching on all nodes
  we simply add $\pmMone \cap \pmMtwo$ to $\pmMprime$ 
  which does not change any of the two required properties.
\end{proof}

Suppose there is a third perfect matching $\pmMthree$ and we want to make $\pmMprime$
adjacent to this matching as well. The remainder of this section
is dedicated to the proof of the following result.

\begin{theorem}
  \label{TheoremThreeCommonNeighbor}
  Let $\pmMone$ and $\pmMtwo$ be two adjacent perfect
  matchings and $\pmMthree$ a third perfect matching.
  Then the three matchings are pairwise adjacent 
  or there exists a perfect matching $\pmMprime$ adjacent to all three.
\end{theorem}

Before we state the proof, we introduce the notion of good perfect matchings.
The first part of the proof is dedicated to proving their existence,
while the second part shows that good perfect matchings, which are
minimal in a certain sense, satisfy the properties claimed by
Theorem~\ref{TheoremThreeCommonNeighbor}.

We first fix some notation for the rest of this section.
Let $\pmMone$, $\pmMtwo$ and $\pmMthree$ be three perfect matchings
such that $\pmMone$ and $\pmMtwo$ are adjacent.
Denote by $V^* := V(\pmMone \Delta \pmMtwo)$ the node set of the single
alternating $\pmMone$-$\pmMtwo$-cycle.

For a perfect matching $\pmMprime$ we denote by \emph{$\pmMthree$-$\pmMprime$-components}
the connected components of $\pmMthree \cup \pmMprime$
and by $c(\pmMthree,\pmMprime)$ their number.
We call a perfect matching $\pmMprime$ \emph{good} if
the following five properties hold:
\begin{enumerate}[(A)]
\item\label{PropertyMatchingAdjacency}
  $\pmMprime$ is adjacent to $\pmMone$ and $\pmMtwo$.
\item\label{PropertyMatchingCyclesTouch}
  All $\pmMthree$-$\pmMprime$-components touch the node-set $V^*$ of $\pmMone \Delta \pmMtwo$.
\item\label{PropertyMatchingCycleEdges}
  All $\pmMprime$-edges which also belong to the $\pmMone$-$\pmMtwo$-cycle, i.e., the edges from $\pmMprime \cap (\pmMone \Delta \pmMtwo)$,
  are contained in the same $\pmMthree$-$\pmMprime$-component.
\item\label{PropertyMatchingComponentSum}
  $\pmMthree \neq \pmMprime$ and $c(\pmMthree,\pmMprime) \leq \frac{1}{2}|\pmMone \Delta \pmMprime| + \frac{1}{2}|\pmMtwo \Delta \pmMprime| - 3$ holds.
\item\label{PropertyMatchingComponentNodesCovered}
  $c(\pmMthree,\pmMprime) \leq \frac{1}{2}|\pmMj \Delta \pmMprime|$ holds for $j=1,2$
  and equality holds only if we have $V(\pmMk \Delta \pmMprime) \supseteq V^*$
  for $k = 3 - j$, i.e., $\setdef{\pmMj,\pmMk} = \setdef{\pmMone,\pmMtwo}$.
\end{enumerate}

We first establish the existence of good perfect matchings.

\begin{lemma}
  \label{TheoremGoodMatchingExists}
  Let $\pmMone$, $\pmMtwo$, $\pmMthree$ be three perfect
  matchings of $K_{2n}$ such that $\pmMone \cap \pmMtwo \cap \pmMthree = \emptyset$ holds
  and such that $\pmMone$ and $\pmMtwo$ are adjacent, but $\pmMthree$
  is not adjacent to both of them.
  
  Then there exists a good perfect matching $\pmMprime$.
\end{lemma}

\begin{proof}
  Let $\pmMbar$ be the perfect matching adjacent to $\pmMone$ and $\pmMtwo$
  constructed in Lemma~\ref{TheoremAdjacentCycle}.

  Note that it satisfies $\pmMbar \cap (\pmMone \Delta \pmMtwo) = \emptyset$
  as well as $|\pmMone \Delta \pmMbar| = |\pmMtwo \Delta \pmMbar| = |\pmMone \Delta \pmMtwo| = |V^*| \geq 4$.
  We now enlarge the $M_i$-$\pmMbar$-cycles ($i=1,2$) in order
  to remove $\pmMthree$-$\pmMbar$-cycles which do not touch $V^*$ in order to satisfy Property~\eqref{PropertyMatchingCyclesTouch}.

  Let $\setdef{u_0, v_0}$ be an $\pmMbar$-edge with $u_0, v_0 \in V^*$.
  Let $C_1, C_2, \ldots, C_s$ be all $\pmMthree$-$\pmMbar$-cycles with $V(C_i) \cap V^* = \emptyset$
  and let, for $i=1,2,\ldots,s$, $\{u_i, v_i\} \in C_i \cap \pmMbar$ be any $\pmMbar$-edge of $C_i$.
  Define $\pmMprime$ to be
  \begin{equation}
    \begin{array}{rcl}
      \pmMprime & := & \left(\pmMbar \setminus 
          \setdef{\setdef{u_i,v_i}}[i=0,1,\ldots,s] \right) \\
              &    & \cup~ \setdef{\setdef{u_i, v_{i+1}}}[i=0,1,\ldots,s]
    \end{array}
  \end{equation}
  where $v_{s+1} = v_0$ (see Figure~\ref{FigureConnectOuterCycles}).

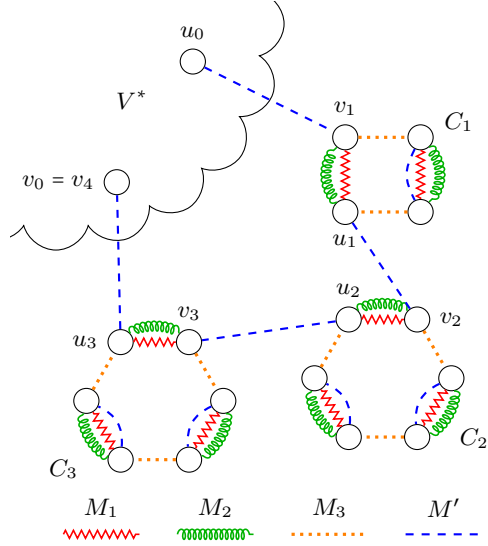
\begin{figure}[ht]
  \begin{center}
    \begin{tikzpicture}
      \piInputConfigured{TikZ-matchings-common-neighbor.tex}%
        [job=connect-outer-cycles,\tikzOptions]
    \end{tikzpicture}
    \begin{tikzpicture}
      \piInputConfigured{TikZ-matchings-common-neighbor.tex}%
        [job=legend existence,\tikzOptions]
    \end{tikzpicture}
  \end{center}
  \caption{Construction in Lemma~\ref{TheoremGoodMatchingExists} with $3$ outer cycles.}
  \label{FigureConnectOuterCycles}
\end{figure}

  We now verify Property~\eqref{PropertyMatchingAdjacency}, i.e., that $\pmMprime$ is adjacent to $M_i$ ($i=1,2$).
  Since the cycles $C_1, \ldots, C_s$ do not touch $V^*$,
  $\pmMbar$ and $M_i$ coincide outside $V^*$.
  Hence, the modification replaces the $\pmMbar$-edge $\setdef{u_0, v_0}$
  by an alternating $\pmMbar$-$M_i$-path from $u_0$ to $v_0$
  which visits exactly $2$ nodes of each $C_i$, thus indeed $\pmMprime$ is adjacent to both $\pmMone$ and $\pmMtwo$.

  In order to prove Properties~\eqref{PropertyMatchingCyclesTouch} and \eqref{PropertyMatchingComponentNodesCovered},
  let us check how the $\pmMthree$-$\pmMprime$-components look like.
  The $\pmMthree$-$\pmMprime$-cycle constructed above contains nodes $u_0$ and $v_0$.
  All other $\pmMthree$-$\pmMprime$-cycles were also $\pmMthree$-$\pmMbar$-cycles,
  and hence, by definition of the $C_i$ above, have at least two nodes of $V^*$
  in common since one of their $\pmMbar$-edges has both endpoints in $V^*$.
  All edges in $\pmMthree \cap \pmMprime$ must also lie in $V^*$
  since outside $V^* \cup V(C_1) \cup \ldots V(C_s)$ the matchings $\pmMprime$, $\pmMone$ and $\pmMtwo$
  are the same and $\pmMone \cap \pmMtwo \cap \pmMthree = \emptyset$ holds.

  Hence all $\pmMthree$-$\pmMprime$-components have at least two nodes in $V^*$,
  in particular, Property~\eqref{PropertyMatchingCyclesTouch} holds.
  It also proves the inequality statement of Property~\eqref{PropertyMatchingComponentNodesCovered}
  because $V(\pmMj \Delta \pmMprime) \supseteq V^*$ for $j=1,2$.
  Furthermore, the containment statement is due to Lemma~\ref{TheoremAdjacentCycle},
  since we have $V(\pmMk \Delta \pmMprime) \supseteq V(\pmMk \Delta \pmMbar) = V^*$ for $k=1,2$.

  In order to verify Property~\eqref{PropertyMatchingCycleEdges},
  observe that $\pmMbar$ satisfies $\pmMbar \cap (\pmMone \Delta \pmMtwo) = \emptyset$.
  Since all edges in $\pmMprime$ that were not in $\pmMbar$ have at least one endpoint
  outside $V^*$, we also have $\pmMprime \cap (\pmMone \Delta \pmMtwo) = \emptyset$.
  Hence, Property~\eqref{PropertyMatchingCycleEdges} is satisfied trivially.

  It remains to show that Property~\eqref{PropertyMatchingComponentSum} holds.
  Clearly, since $\pmMthree$ is adjacent to at most one of $\pmMone$, $\pmMtwo$,
  we have $\pmMthree \neq \pmMprime$.
  Since we established as part of Property~\eqref{PropertyMatchingComponentNodesCovered},
  that $c(\pmMthree,\pmMprime) \leq \frac{1}{2} |\pmMone \Delta \pmMprime|$ holds,
  it suffices to show that $|\pmMtwo \Delta \pmMprime|$ is at least $6$.
  Suppose, for the sake of contradiction, that this is not the case, i.e., $|\pmMtwo \Delta \pmMprime| \leq 4$ holds,
  which in turn implies $c(\pmMthree,\pmMprime) \leq 2$.
  Also $|\pmMtwo \Delta \pmMprime| \geq 4$ holds since both matchings are adjacent.
  This implies that we have equality in the containment $V(\pmMtwo \Delta \pmMprime) \supseteq V(\pmMtwo \Delta \pmMbar) = V^*$,
  from which we conclude that $s=0$ holds, i.e., $\pmMprime = \pmMbar$.
  These properties already prove that the $\pmMprime$-edges which match the nodes of $V^*$ are 
  exactly the two chords of the $\pmMone$-$\pmMtwo$-cycle (see Figure~\ref{FigureShort}).
  It is now easy to verify that then $\pmMone$, $\pmMtwo$, and $\pmMthree$ must be pairwise
  adjacent, a contradiction to the assumptions of this lemma.
  \begin{figure}[H]
    \begin{center}
      \begin{tikzpicture}
        \piInputConfigured{TikZ-matchings-common-neighbor.tex}%
          [job=short,\tikzOptions]
      \end{tikzpicture} \\
      \begin{tikzpicture}
      \piInputConfigured{TikZ-matchings-common-neighbor.tex}%
        [job=legend existence,\tikzOptions]
    \end{tikzpicture}
    \end{center}
    \caption{A special case where $\pmMthree$ is adjacent to $\pmMone$ and $\pmMtwo$.}
    \label{FigureShort}
  \end{figure}
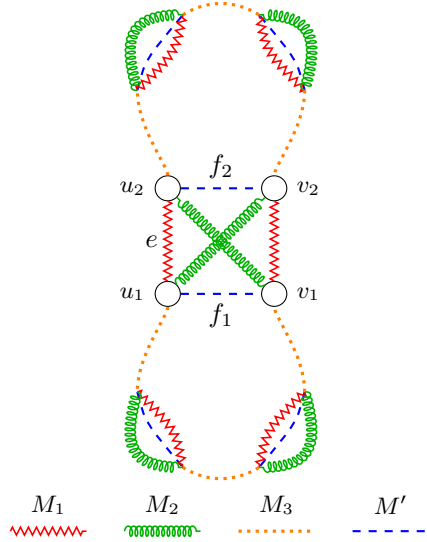
\end{proof}

\begin{proof}[Proof of Theorem~\ref{TheoremThreeCommonNeighbor}]
  Let $\pmMone, \pmMtwo, \pmMthree$ be as stated in the Theorem.
  We assume, without loss of generality, that $\pmMone \cap \pmMtwo \cap \pmMthree = \emptyset$ holds, since
  otherwise we can restrict ourself to the graph with the nodes of this set deleted.
  We also assume that $\pmMone$, $\pmMtwo$, and $\pmMthree$ are not pairwise adjacent,
  since otherwise there is nothing to prove.

  In this situtation, Lemma~\ref{TheoremGoodMatchingExists} guarantees the existence of a good perfect matching.
  Let $\pmMprime$ be a good perfect matching with minimum $c(\pmMthree,\pmMprime)$.
  If $c(\pmMthree, \pmMprime) = 1$ holds, $\pmMprime$ is adjacent to $\pmMthree$
  (and also adjacent to $\pmMone$ and $\pmMtwo$ by Property~\eqref{PropertyMatchingAdjacency}) and we are done.
  Hence, for the sake of contradiction, we from now on assume that $c(\pmMthree,\pmMprime)$ is at least $2$.
  The strategy is to construct another good matching $\pmMstar$ with $c(\pmMthree,\pmMstar) < c(\pmMthree,\pmMprime)$.

  Due to Property~\eqref{PropertyMatchingCycleEdges}, there exists an $\pmMthree$-$\pmMprime$-component $\widehat{C}$
  containing all edges (if any) from $\pmMprime \cap (\pmMone \Delta \pmMtwo)$.
  $\pmMone \Delta \pmMtwo$ is a single cycle
  visiting all nodes in $V^*$ all of which are in some
  $\pmMthree$-$\pmMprime$-component as they are matched by $\pmMprime$.
  Thus, by Property~\eqref{PropertyMatchingCyclesTouch}, the component
  $\widehat{C}$ is connected to at least one other $\pmMthree$-$\pmMprime$-component
  by an edge from $\pmMone$-$\pmMtwo$.
  Let us choose such an edge $e \in \pmMj \setminus \pmMk$ for some $j=1,2$ and $k=3-j$,
  and if such edges exist for both values of $j$, choose $e$ such that $|\pmMj \Delta \pmMprime|$ is maximum.

  We claim that $|\pmMj \Delta \pmMprime| \geq 6$ holds.
  Suppose, for the sake of contradiction, $|\pmMj \Delta \pmMprime| = 4$.
  Property~\eqref{PropertyMatchingComponentNodesCovered} implies
  that $c(\pmMthree,\pmMprime) \leq 2$ holds. 
  Since $c(\pmMthree,\pmMprime) \geq 2$ also holds, we have equality and
  then Property~\eqref{PropertyMatchingComponentNodesCovered} implies
  that the $\pmMk$-$\pmMprime$-cycle 
  covers all nodes in $V^*$.
  This cycle connects the only two $\pmMthree$-$\pmMprime$-components $\widehat{C}$ and $C'$ via
  at least two $\pmMk$-edges $f,f'$ since the $\pmMprime$-edges of the cycle are inside their respective components.
  Note that $|\pmMk \Delta \pmMprime| \geq 2 c(\pmMthree,\pmMprime) + 6 - |\pmMj \Delta \pmMprime| \geq 6$ holds
  by Property~\eqref{PropertyMatchingComponentSum}.
  Hence, by the maximality assumption for the choice of edge $e$, this implies
  that there is no edge in $\pmMk \setminus \pmMj$ which connects $\widehat{C}$ to $C'$.
  Since $f,f' \in \pmMk$ both connect $\widehat{C}$ to $C'$, it follows that $f,f' \in \pmMj$ holds as well.
  Because $|\pmMj \Delta \pmMprime| = 4$ holds we have $\pmMj \setminus \pmMprime = \setdef{f,f'}$.
  Hence, the alternating $\pmMj$-$\pmMprime$-cycle of length $4$ is also an alternating
  $\pmMk$-$\pmMprime$-cycle. But the latter has at least length $6$ as argued above which yields a contradiction.

  To summarize, we now have two distinct $\pmMthree$-$\pmMprime$-components $\widehat{C}$ and $C'$
  connected by an edge $e \in \pmMj \setminus \pmMk$ for some $j = 1,2$ such that
  $|\pmMj \Delta \pmMprime| \geq 6$ holds and
  all edges from $\pmMprime \cap (\pmMone \Delta \pmMtwo)$ (if any) are in $\widehat{C}$.

  Let $u_1 \in V(\widehat{C})$ and $u_2 \in V(C')$ be the endpoints of edge $e$.
  Let $f_1 = \setdef{u_1, v_1}, f_2 = \setdef{u_2, v_2} \in \pmMprime$
  be the edges matching $u_1$ and $u_2$.
  We clearly have $v_1 \in V(\widehat{C})$ and $v_2 \in V(C')$ as well
  since $f_1$ and $f_2$ are contained in their respective $\pmMthree$-$\pmMprime$-components.
  In particular we have $f_1, f_2 \neq e$ and $u_1,u_2,v_1,v_2$ are pairwise distinct nodes.
  Because $f_2 \notin \pmMj$ ($e \in \pmMj$ and $e \cap f_2 \neq \emptyset$) and $f_2 \not\subseteq V(\widehat{C})$ hold,
  Property~\eqref{PropertyMatchingCycleEdges} implies $f_2 \notin \pmMk$ 
  (note that $\widehat{C}$ is the component mentioned in Property~\eqref{PropertyMatchingCycleEdges}).

  If also $f_1 \notin \pmMk$ holds, $f_1$ and $f_2$ belong to $\pmMk \Delta \pmMprime$ which is a single cycle
  by Property~\eqref{PropertyMatchingAdjacency}, and hence there exists a walk $W$
  on $\pmMk \Delta \pmMprime$ starting in $u_1$ with edge $f_1$ which visits nodes $u_2$ and $v_2$ (in some order).

  We are now ready to create a new good perfect matching $\pmMstar$ related to $\pmMprime$ by small changes.
  For this, we distinguish two cases. 
  For each case we establish Property~\eqref{PropertyMatchingAdjacency} separately, and afterwards
  prove the remaining properties for both cases in parallel.

\medskip
  \textbf{Case 1:} $f_1 \notin \pmMk$ holds and $u_2$ comes before $v_2$ on walk $W$.

\medskip
  Let
  $\pmMstar := \left( \pmMprime \setminus \setdef{f_1,f_2} \right) 
      \cup \setdef{\setdef{u_1,u_2}, \setdef{v_1,v_2}}$ (see Figure~\ref{FigureUglyExchange}).
\smallskip

  We now prove Property~\eqref{PropertyMatchingAdjacency} for $\pmMstar$.
  The symmetric difference $\pmMk \Delta \pmMstar$ consists of a single cycle
  that arises from the cycle $\pmMk \Delta \pmMprime$ by removing edges $f_1, f_2$
  and adding edges $\setdef{u_2,u_1}$ and $\setdef{v_2,v_1}$.
  For $\pmMj$, the situation is different, since there is a new component $e \in \pmMj \cap \pmMstar$.
  The new $\pmMj$-$\pmMstar$-cycle is now $2$ edges shorter than the $\pmMj$-$\pmMprime$-cycle was before
  since it visits the edge $\setdef{v_1, v_2}$ instead of the path $v_1-u_1-u_2-v_2$.
  But because we ensured $|\pmMj \Delta \pmMprime| \geq 6$ before, we have
  $|\pmMj \Delta \pmMstar| \geq 4$, that is, $\pmMstar$ is also adjacent to $\pmMj$.

  \begin{figure}[H]
    \begin{center}
      \begin{tikzpicture}
        \piInputConfigured{TikZ-matchings-common-neighbor.tex}%
          [job=ugly-exchange,\tikzOptions]
      \end{tikzpicture}
      \begin{tikzpicture}
        \piInputConfigured{TikZ-matchings-common-neighbor.tex}%
          [job=legend switchings,\tikzOptions]
      \end{tikzpicture}
    \end{center}
    \caption{Modifications in Case~1.}
    \label{FigureUglyExchange}
  \end{figure}
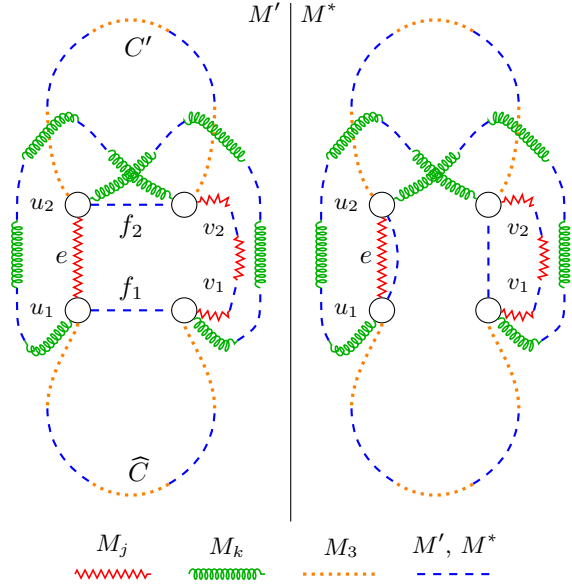
  
\medskip
  \textbf{Case 2:} $f_1 \in \pmMk$ holds or $f_1 \notin \pmMk$ and $u_2$ comes after $v_2$ on walk $W$.

\medskip
  Let
    $\pmMstar := \left( \pmMprime \setminus \setdef{f_1,f_2} \right) 
      \cup \setdef{\setdef{u_1,v_2}, \setdef{u_2,v_1}}$ (see Figure~\ref{FigureSimpleExchange}).
\smallskip
  
  We now prove Property~\eqref{PropertyMatchingAdjacency} for $\pmMstar$.
  The symmetric difference $\pmMk \Delta \pmMstar$ consists of a single cycle
  that arises from the cycle $\pmMk \Delta \pmMprime$ by removing edges $f_1, f_2$
  and adding edges $\setdef{v_2,u_1}$ and $\setdef{u_2,v_1}$.
  There is also only one $\pmMj$-$\pmMstar$-cycle which is essentially equal to the $\pmMj$-$\pmMprime$-cycle,
  except that the path $v_1-u_1-u_2-v_2$ was replaced by the path $v_1-u_2-u_1-v_2$.
  
  \begin{figure}[H]
    \begin{center}
      \begin{tikzpicture}
        \piInputConfigured{TikZ-matchings-common-neighbor.tex}%
          [job=simple-exchange,\tikzOptions]
      \end{tikzpicture}
      \begin{tikzpicture}
        \piInputConfigured{TikZ-matchings-common-neighbor.tex}%
          [job=legend switchings,\tikzOptions]
      \end{tikzpicture}
    \end{center}
    \caption{Modifications in Case~2.}
      \label{FigureSimpleExchange}
  \end{figure}
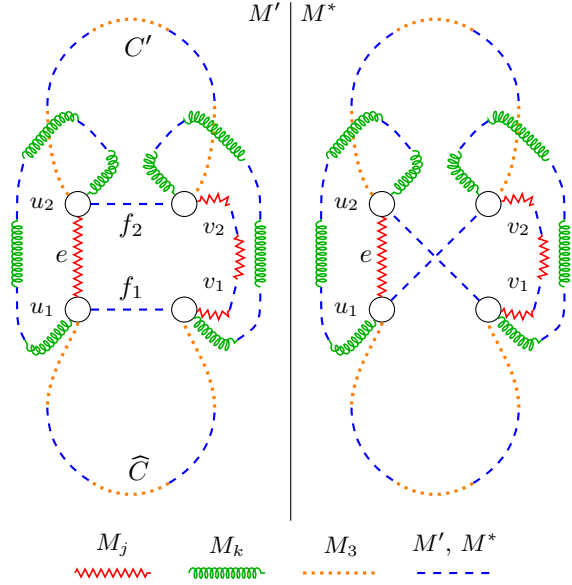

\medskip
  In Case~1 as well as in Case~2, $\pmMstar$ is again a perfect matching since $\pmMprime$ was a perfect matching
  and they differ only in the way the nodes $u_1, u_2, v_1, v_2$ are matched.
  Furthermore, $\pmMstar$ connects the two components $\widehat{C}$ and $C'$, that is,
  $c(\pmMthree,\pmMstar) = c(\pmMthree,\pmMprime) - 1$. 
  In order to create the desired contradiction to the minimality of $c(\pmMthree, \pmMprime)$,
  it remains to prove that $\pmMstar$ satisfies 
  Properties~\eqref{PropertyMatchingCyclesTouch},\eqref{PropertyMatchingCycleEdges},
  \eqref{PropertyMatchingComponentSum} and \eqref{PropertyMatchingComponentNodesCovered}.
  
  Property~\eqref{PropertyMatchingCyclesTouch} is satisfied for $\pmMstar$ because
  $V(\widehat{C})$ is contained in an $\pmMthree$-$\pmMstar$-component and
  all edges in $\pmMstar \setminus \pmMprime$ are contained in the same component.

  Property~\eqref{PropertyMatchingCycleEdges} is also satisfied for $\pmMstar$ since all $\pmMstar$-edges that were not
  $\pmMprime$-edges before, are contained in cycle $\widehat{C}$ (which was by definition the only
  $\pmMthree$-$\pmMprime$-cycle containing edges in $\pmMprime \cap (\pmMone \Delta \pmMtwo)$).

  We now prove that Property~\eqref{PropertyMatchingComponentSum} is satisfied for $\pmMstar$. We have
  \begin{align*}
    c(\pmMthree,\pmMstar) + 3 
    &= c(\pmMthree,\pmMprime) + 3 - 1
    \leq \frac{1}{2}|\pmMone \Delta \pmMprime| + \frac{1}{2}|\pmMtwo \Delta \pmMprime| -1\\
    &= \frac{1}{2}|\pmMk \Delta \pmMprime| + \frac{1}{2}|\pmMj \Delta \pmMprime| -1 \\
    &\leq \frac{1}{2}|\pmMk \Delta \pmMstar| + \frac{1}{2}\left(|\pmMj \Delta \pmMstar| + 2\right) -1 \\
    &= \frac{1}{2}|\pmMone \Delta \pmMstar| + \frac{1}{2}|\pmMtwo \Delta \pmMstar| \ ,
  \end{align*}
  where the first inequality is due to Property~\eqref{PropertyMatchingComponentSum} for $\pmMprime$ and
  the last inequality comes from the fact that in Case~1, $\pmMj \Delta \pmMstar$ has two fewer edges than $\pmMj \Delta \pmMprime$
  and in Case~2, the cardinalities agree.
  
  By similar arguments, $c(\pmMthree,\pmMstar) \leq \frac{1}{2} |M_i \Delta \pmMstar|$ holds for $i=1,2$.
  In order to prove that Property~\eqref{PropertyMatchingComponentNodesCovered} is satisfied for $\pmMstar$,
  assume that $c(\pmMthree,\pmMstar) = \frac{1}{2} |M_i \Delta \pmMstar|$ holds for some $i \in \setdef{1,2}$.
  Due to $c(\pmMthree,\pmMstar) = c(\pmMthree,\pmMprime) - 1$,
  this implies $i = j$ and we are in Case~1 since only there $|M_i \Delta \pmMstar|$ is less than $|M_i \Delta \pmMprime|$,
  and also have $c(\pmMthree,\pmMprime) = \frac{1}{2} |\pmMj \Delta \pmMprime|$.
  Property~\eqref{PropertyMatchingComponentNodesCovered} of $\pmMprime$ guarantees that $V(\pmMk \Delta \pmMprime) \supseteq V^*$ holds. 
  But since the node sets of $\pmMk \Delta \pmMprime$ and $\pmMk \Delta \pmMstar$ are the same, we also have $V(\pmMk \Delta \pmMstar) \supseteq V^*$.
 
  We proved that $\pmMstar$ is a good perfect matching, yielding the required contradiction to the minimality assumption of $c(\pmMthree,\pmMprime)$
  which completes the proof.
\end{proof}

\section{A Related Question}

Let us make a brief digression on the potential relevance of simple extensions with
respect to questions related to the diameter of a polytope, i.e.,
the maximum distance (minimum number of edges on a path) between any pair of vertices
in the graph of the polytope. We denote by $\Delta(d,m)$ the maximum diameter
of any $d$-dimensional polytope with~$m$ facets. It is well-known that
$\Delta(d,m)$ is attained by simple polytopes.
A necessary condition for a polynomial time variant of the simplex-algorithm to exist
is that $\Delta(d,m)$ is bounded by a polynomial in~$d$ and $m$
(thus by a polynomial in~$m$). 
In fact, in 1957 Hirsch even conjectured (see~\cite{Dantzig63})
that $\Delta(d,m)\le m-d$ holds,
which has only rather recently been disproved by Santos~\cite{Santos12}.
However, still it is even unknown whether $\Delta(d,m)\le 2m$ holds true, and the question,
whether $\Delta(d,m)$ is bounded polynomially
(i.e., whether the \emph{polynomial Hirsch-conjecture} is true) is a major open
problem in Discrete Geometry. 

In view of the fact that linear optimization over a polytope can be performed 
by linear optimization over any of its extensions, a reasonable relaxed version
of that question might be to ask whether every $d$-dimensional polytope~$P$ with~$m$ facets
admits an extension whose size and diameter both are bounded polynomially in~$m$.
Stating the relaxed question in this naive way, the answer clearly is positive,
as one may construct an extension by forming a pyramid over~$P$ 
(after embedding~$P$ into $\R^{\dim[P]+1}$), which has diameter two.
However, in some accordance with the way the simplex algorithm works by pivoting
between bases rather than only by proceeding along edges, it seems to make sense
to require the extension to be simple (which a pyramid, of course, in general is not).
But still, this is not yet a useful variation, 
since our result on flow polytopes shows that there are polytopes
that even do not admit a polynomial (in the number of facets) size simple extension at all.
Therefore, we propose to investigate the following question,
whose positive answer would be implied by establishing
the polynomial Hirsch-conjecture (as every polytope is an extension of itself).  
\begin{question}
  Does there exist a polynomial $q$ such that every \emph{simple} polytope $P$
  with $m$ facets has a \emph{simple} extension $Q$ with at most $q(m)$ many facets
  and diameter at most $q(m)$?
\end{question}

\textbf{Acknowledgements.} We are greatful to the referees whose comments lead to
significant improvements in the presentation of the material.

\pagebreak[3]
\addcontentsline{toc}{chapter}{Bibliography}

\bibliographystyle{plain}
\bibliography{references}

\end{document}